\documentclass[bj]{imsart}

\RequirePackage{amsthm,amsmath,amsfonts,amssymb}
\RequirePackage[numbers]{natbib}
\RequirePackage[colorlinks,citecolor=blue,urlcolor=blue]{hyperref}%% uncomment this for coloring bibliography citations and linked URLs

\startlocaldefs
\theoremstyle{plain}
\newtheorem{theorem}{Theorem}[section]
\newtheorem{lemma}[theorem]{Lemma}
\newtheorem{corollary}[theorem]{Corollary}
\theoremstyle{remark}
\newtheorem*{remark}{Remark}
\newtheorem{definition}[theorem]{Definition}
\newtheorem{example}{Example}
\newtheorem{assumption}[theorem]{Assumption}
\newtheorem*{cond1}{Condition $\mathcal{D}(x_n;K_n)$}
\newtheorem*{cond2}{Condition $\mathcal{D}^\ell(x_n;k_n;\theta)$}
\newtheorem*{cond3}{Condition $\mathcal{D}^{(m)}(x_n;k_n)$}
\newtheorem*{cond4}{Condition $\overline{\mathcal{D}}(x_n;k_n;K_n)$}
\usepackage{mathtools}
\usepackage{enumitem}

\DeclarePairedDelimiter\abs\lvert\rvert

\DeclareMathOperator*{\argmin}{\arg\min}

\newcommand{\NN}{\mathbb{N}}
\newcommand{\ZZ}{\mathbb{Z}}

\newcommand{\Zd}{{\ZZ^d}}
\newcommand{\vZ}{{v\in \Zd}}
\newcommand{\RR}{\mathbb{R}}
\newcommand{\Rd}{{\RR^d}}
\newcommand{\vR}{{v\in \Rd}}
\newcommand{\PP}{\mathbb{P}}
\newcommand{\EE}{\mathbb{E}}

\newcommand{\dd}{\mathrm{d}}
\newcommand{\Jz}{J_z^{n,k_n}}
\newcommand{\dnm}{D_{n,k_n}^{-}}
\newcommand{\dnp}{D_{n,k_n}^{+}}
\newcommand{\Bb}{\mathcal{B}}
\newcommand{\I}[1]{\boldsymbol{1}_{#1}}
\newcommand{\FF}{\overline F}
\newcommand{\NNN}{\overline N}
\newcommand{\bb}[1]{\boldsymbol{#1}}
\newcommand{\sca}{\bb{ c_n}}
\newcommand{\scal}[1]{ c_{n,#1}}

\endlocaldefs

\begin{document}

\begin{frontmatter}
\title{Extremal clustering and cluster counting for spatial random fields}
%\title{A sample article title with some additional note\thanksref{T1}}
\runtitle{Extremal clustering of random fields}
%\thankstext{T1}{A sample of additional note to the title.}

\begin{aug}
\author[A]{\inits{AR}\fnms{Anders} \snm{R{\o}nn--Nielsen}\ead[label=e1,mark]{aro.fi@cbs.dk}}
\and
\author[A]{\inits{MS}\fnms{Mads} \snm{Stehr}\ead[label=e2,mark]{mast.fi@cbs.dk}}
%%%%%%%%%%%%%%%%%%%%%%%%%%%%%%%%%%%%%%%%%%%%%%
%% Addresses                                %%
%%%%%%%%%%%%%%%%%%%%%%%%%%%%%%%%%%%%%%%%%%%%%%
\address[A]{Department of Finance, Center for Statistics,
Copenhagen Business School, Frederiksberg, Denmark.
\printead{e1,e2}}
\end{aug}

\begin{abstract}
We consider a stationary random field indexed by an increasing sequence of subsets of $\Zd$ obeying a very broad geometrical assumption on how the sequence expands. Under certain mixing and local conditions, we show how the tail distribution of the individual variables relates to the tail behavior of the maximum of the field over the index sets in the limit as the index sets expand.

Furthermore, in a framework where we let the increasing index sets be scalar multiplications of a fixed set $C$, potentially with different scalars in different directions, we use two cluster definitions to define associated cluster counting point processes on the rescaled index set $C$; one cluster definition divides the index set into more and more boxes and counts a box as a cluster if it contains an extremal observation. The other cluster definition that is more intuitive considers extremal points to be in the same cluster, if they are close in distance. We show that both cluster point processes converge to a Poisson point process on $C$. Additionally, we find a limit of the mean cluster size. Finally, we pay special attention to the case without clusters.
\end{abstract}

\begin{keyword}
\kwd{Extreme value theory}
\kwd{spatial models}
\kwd{random fields}
\kwd{intrinsic volumes}
\kwd{extremal index}
\kwd{cluster counting process}
\kwd{limit theorems}
\end{keyword}

\end{frontmatter}

%%%%%%%%%%%%%%%%%%%%%%%%%%%%%%%%%%%%%%%%%%%%%%
%%%% Main text entry area:

%----------------------------
	\section{Introduction}
%----------------------------
This paper will provide a multitude of results within the realm of spatial extreme value theory, some of which will be generalizations of one-dimensional results or results from a simpler spatial framework. More precisely, we consider a stationary field $(\xi_v)_\vZ$ where $d\in\NN$, for which we consider the extremal behavior of $(\xi_v)_{v\in D_n}$ under a very rich asymptotic regime of increasing index sets $D_n \subseteq \Zd$. For detailed treatments of classical extreme value theory and its generalizations to stationary sequences in the one-dimensional case, we refer to e.g. \cite{Embrecths1997} and \cite{Leadbetter1983}.

In the literature, results for spatial objects, comparable to some of the present results, are to the best of the authors' knowledge only formulated under the assumption of $(D_n)_{n\in\NN}$ being a sequence of increasing boxes; see for instance \cite{Jakubowski2019,Soja2019} that provides results like the ones found in Section~\ref{sec:independenceapproximation} below under this additional geometrical assumption. In contrast, we allow the index sets $D_n$ to expand in a much more general way; we refer to the authors' papers \cite{StehrRonnNielsen2021a,StehrRonnNielsen2020,StehrRonnNielsen2021b} for similar 
however slightly less general
assumptions on the sequence of index sets, but with other and less general results.
Our requirement will simply be that each $D_n$ is the lattice points of a sufficiently nice full-dimensional subset of $\Rd$.
To be precise, we assume that $D_n$ is given as $D_n=C_n\cap \Zd$, where $C_n$ is a union of convex bodies
contained in an expanding box. This essentially means that all of the convex bodies in the union, when scaled coordinate-wise with the side lengths of the surrounding box, has bounded intrinsic volumes as $n\to\infty$;
see \cite[Chapter~4]{Schneider1993} for an introduction to convex geometry and in particular intrinsic volumes of convex bodies, and see Assumption~\ref{ass:Cnassumption} below for the formal description. If $d=3$ and $C_n$ itself is a convex body for all $n\in\NN$, 
the assumptions contains --- but is indeed not limited to --- the case where $C_n$ expands at a similar pace in all directions, meaning equivalently that the mean width of $C_n$ is asymptotically bounded by the cubic root of its volume, and that the surface area is asymptotically bounded by the cubic root of the squared volume.

The first part of the paper concerns the asymptotic distribution function of $\max_{v\in D_n} \xi_v$ along some sequence $x_n$. More precisely, we first formulate a mixing condition ensuring appropriate asymptotic independence of the field, which implies the following representation as $n\to\infty$:
\begin{equation}\label{eq:intro1}
	\PP \bigl(\max_{v\in D_n} \xi_v \le x_n \bigr) 
	=
	\exp \Bigl( - \abs{D_n} 
	\PP\bigl(\max_{v\in A_0^{n,k_n}} \xi_v\le x_n < \xi_0 \bigr) \Bigr) + o(1) .
\end{equation}
Here the set $A_0^{n,k_n}$ is defined by
\[
	A_0^{n,k_n} = \bigl\{v\in \bigl(\bigtimes_{\ell=1}^d [-t_{n,k_n}^\ell,-t_{n,k_n}^\ell] \bigr)\cap \Zd\::\: 0 \prec v\bigr\} ,
\]
where $t_{n,k_n}^\ell\in \NN$ is an increasing integer for all $\ell=1,\dots,d$, and $\preceq$ is a translation invariant total order on $\Zd$ (as for instance the lexicographical order). A one-dimensional (i.e. $d=1$) counterpart of this result can be found in \cite[Theorem~2.1]{Obrien1987}, and in \cite[Theorem~3.1]{Soja2019} a result is obtained in the special case where $(D_n)$ is an increasing sequence of boxes in $\Zd$. Secondly, we formulate a local condition describing the clustering of exceedances of $(\xi_v)_\vR$ over the threshold $(x_n)$ in terms of an index $\theta\in (0,1]$; the mixing and local conditions are named $\mathcal D(x_n;K_n)$ and $\mathcal{D}^\ell(x_n;k_n;\theta)$, respectively (with $K_n$ and $k_n$ specified in the relevant section), and will be referred to as this in the present introduction. If e.g. $(\xi_v)_\vZ$ is $m$-dependent, meaning in particular that any two variables a distance more than $m$ apart are independent, the local condition $\mathcal{D}^\ell(x_n;k_n;\theta)$ simply reads
\[
	\lim_{n\to\infty} \PP( \max_{v \in A_0^{(m)}} \xi_v \le x_n \mid \xi_0 > x_n) = \theta ,
\]
where the set $A_0^{(m)} = \{v\in [-m,m]^d \cap \Zd\::\: 0 \prec v\}$ is a fixed subset of $A_0^{n,k_n}$ from the representation \eqref{eq:intro1}. A simple example given in the paper is the stationary field defined by
\[
	\xi_v = \max_{z \in v + B} Y_z , \qquad v\in\Zd,
\]
where $B$ is a finite subset of $\Zd$ and the field $(Y_z)_{z\in\Zd}$ consists of i.i.d. variables. In fact, the field is $m$-dependent for sufficiently large $m$ and it satisfies the local condition with $\theta=1/\abs{B}$, where $\abs{B}$ denotes the number of points in $B$. Using the same cardinality notation as in the example, we show under the mixing and local conditions that
\begin{equation}\label{eq:intro2}
	\abs{D_n} \PP(\xi_0>x_n) \to \tau
	\quad \text{if and only if}\quad
	\PP(\max_{v\in D_n} \xi_v \le x_n) \to \exp(-\theta \tau) 
\end{equation}
as $n\to\infty$ for $\tau \in [0,\infty)$.
This result constitutes a substantial generalization of \cite[Theorem~5]{StehrRonnNielsen2020} that only considers the special case of $\theta=1$ --- meaning that no clustering of exceedances occur --- but under
a similar, but less general, asymptotic expansion of the index sets.
An immediate consequence of the considerations leading to the equivalence above, is the fact that $(\xi_v)$ has extremal index $\theta \in (0,1]$ if and only if the local condition $\mathcal{D}^\ell(x_n;k_n;\theta)$ is satisfied; we refer to Section~\ref{sec:independenceapproximation} for the formal definition of an extremal index (with respect to $D_n$), which is in accordance with the definition given in e.g. \cite{Leadbetter1983} for $d=1$. This result generalizes the claims in \cite{Chernick1991, Soja2019}, dealing with the one-dimensional case and the case of index sets being boxes, respectively, to the present framework allowing for a much more involved asymptotic development of the index sets.

The remainder of the paper is devoted to deriving asymptotic representations of various forms of cluster and exceedance counting processes for spatial random fields satisfying mixing and local conditions as above. Specifically, we work under the additional assumption that
\begin{equation}\label{fml:multipliedset}
	D_n = (\sca C)\cap \Zd 
\end{equation}
where $\sca$ is a $d$-dimensional vector with each entrance tending to $\infty$, and $C \subseteq \Rd$ is a finite union of convex bodies, and standardized together with $(\sca)$ such that $C$ has Lebesgue measure 1. Here, $\sca C$ denotes the coordinate-wise multiplication of $C$ by the entries of $\sca$. This in particular ensures that our sufficient geometric assumption, Assumption~\ref{ass:Cnassumption}, is satisfied. Furthermore we assume that there is a real sequence $(x_n)$ such that the conditions $\mathcal D(x_n;K_n)$ and $\mathcal{D}^\ell(x_n;k_n;\theta)$ are satisfied with $\theta \in (0,1]$, and such that
\[
	\abs{D_n} \PP(\xi_0>x_n) \to \tau \in (0,\infty)
\]
as $n\to\infty$. 

Now, we define two different cluster counting processes of exceedances over the threshold $x_n$. The first is a spatial version of the classical one-dimensional definition, e.g. seen in \cite{Hsing1988} and \cite{Leadbetter1983a}: $\Zd$ is divided into disjoint (increasing) 
boxes 
each potentially being counted as a cluster if there is an $x_n$-exceedance within the 
box. For $d=1$ such a counting process is usually defined on $(0,1]$ via a rescaling of the indices, whereas our geometric
construction in \eqref{fml:multipliedset} allows us to define it on the general set $C$. The second cluster counting process is to the best of our knowledge previously unseen in the literature in relation to extreme value theory. This counting process is also defined on $C$ but is based on the more intuitive definition that a cluster is formed by the indices of $x_n$-exceedances which are within a given distance of other exceedances. Here, the relevant distance corresponds to the length of the diagonal of a rescaled version of the boxes from the first type of counting process. 
We show that both counting processes in fact converges in distribution towards a homogeneous Poisson process on $C$ with intensity $\theta \tau $. A one-dimensional counterpart of the result for the first of the two cluster processes only is found in \cite{Leadbetter1983}. A related spatial result under different and much stronger conditions is found in \cite{Ferreira2012} in the simpler spatial scenario, where the index sets are increasing boxes: It is shown that the original exceedance point process (before collecting them into clusters) converges to a compound Poisson point process. 

Moreover, we show that the expected size of such clusters is asymptotically equal to $1/\theta$. This is in accordance with the typical interpretation of the extremal index as the reciprocal of the mean number of exceedances in a cluster, but it is indeed not a triviality for the second cluster counting process. That the cluster counting processes converge to a Poisson process in particular means independence between cluster positions in the limit. The extremal independence between clusters is further underlined in the additional result, where we demonstrate that the limiting mean cluster size is in fact independent of the total number of clusters in $C$.

Related to the two cluster counting processes defined on $C$, we also define two associated cluster counting processes defined on the original scale in $\Zd$. Using the distributional convergence results above, we show that both of these original-scale counting processes satisfy the following: If $(B_n^1),\dots,(B_n^G)$ are disjoint sequences of subsets of $D_n\subseteq\Zd$ each satisfying the geometric Assumption~\ref{ass:Cnassumption}, then the joint distribution of cluster counts of either type converges in distribution to $(L^1,\dots,L^G)$, which are independent random variables with each $L^g$ being Poisson distributed with parameter $\theta\tau \lim_{n\to\infty} \abs{B_n^g}/\abs{D_n}$ (assuming that the limit exists). This result is far from being a trivial consequence of the convergence of the cluster point processes defined above, as the asymptotic behavior of the $B_n^g$ sets can be considerably more complex than what is obtained by rescaling to original scale of subsets of the set $C$ introduced in \eqref{fml:multipliedset}.

In the last part of the paper, we consider the non-clustering case of $\theta=1$ and the usual point process of exceedances over the threshold $x_n$, now defined on the general set $C$. We show that this process and its associated original-scale process converges exactly as the cluster counting processes but with $\theta=1$; see e.g. \cite[Chapter~5]{Leadbetter1983} for the classical one-dimensional case.

The paper is organized as follows. In Section~\ref{sec:geometry} we introduce the geometric structures applied in the paper. In Section~\ref{sec:independenceapproximation} we introduce the mixing and local conditions, and furthermore show the representation \eqref{eq:intro1} of the distribution function of $\max_{v\in D_n} \xi_v$. This then leads to the equivalence \eqref{eq:intro2} and to the result on the existence of an extremal index. 
Sections~\ref{sec:clusterpointmeasures} and \ref{sec:meannumberpoints} are devoted to the convergence of cluster counting processes and the expected size of such clusters, respectively, and Section~\ref{sec:originalscale} contains convergence results for the cluster counting processes on the original scale $\Zd$. Finally, Section~\ref{sec:noclustering} contains similar results to the ones in Sections~\ref{sec:clusterpointmeasures} and \ref{sec:originalscale} for the non-clustering case.

%----------------------------
	\section{Geometric assumption and preliminaries}\label{sec:geometry}
%----------------------------

As mentioned in the introduction, we consider a stationary random field $(\xi_v)_{v\in \Zd}$ and a sequence of very flexible index sets $(D_n)_{n\in \NN}$ with $D_n \subseteq\Zd$. The main purpose of this section is to present the sufficient assumption on the expansion of the index sets, which is in fact formulated in terms of a continuous counterpart. 

Before presenting the assumption, we mention the following notation used throughout the paper. We let $\abs{{}\cdot{}}$ be a general size-measure understood as follows: $\abs{v}$ is the Euclidean norm of a single one- or multidimensional point $v$, $\abs{A}$ is Lebesgue measure of a full-dimensional set $A\subseteq \Rd$, and $\abs{A}$ is the number of points in a discrete set $A\subseteq \Zd$.
For two sets $A, B\subseteq \Rd$, we define their Minkowski sum by $A\oplus B=\{a+b\mid a\in A,b\in B\}$, and we let $B(r)=\{u\in\Rd\::\: \abs{u}\le r\}$ be the closed ball in $\Rd$ centered at the origin $0\in\Rd$ and with radius $r\ge 0$.
Furthermore, we will use the following notation for a coordinate-wise scaling of a set $A \subseteq \Rd$: With $\bb{r}$ (in bold) denoting a vector of $d$ elements $r_1,\dots,r_d$ we write
\[
	\bb{r} A = \{ (r_1 a_1,\dots,r_d a_d) \in \Rd\::\: (a_1,\dots,a_d)\in A\} . 
\]
That is, $\bb{r} A$ is a compact notation for the linear transformation of $A$ induced by the diagonal matrix with entries $r_1,\dots,r_d$. In particular we obtain for any full-dimensional set $A$ that $\bb{r}A$ has Lebesgue measure
\begin{equation*}%\label{eq:matrixprodleb}
	\abs{\bb{r}A} = \abs{A} \prod_{\ell=1}^d r_\ell  .
\end{equation*}
If $\bb{r}$ is a vector of $d$ identical entries $r$, that is, if the scaling is the same in all directions, we write
\[
	\bb{r}A = rA = \{ (r a_1,\dots,r a_d) \in \Rd\::\: (a_1,\dots,a_d)\in A\} .
\] 
Lastly, for a vector $\bb{r}$ we write $\bb{r}^{-1}$ or $1/\bb{r}$ for the vector of elements $r_1^{-1},\dots,r_d^{-1}$.

First, we define a type of continuous set which will be the cornerstone of the index set assumption. In the following definition, a convex body in $\Rd$ is a compact and convex set with non-empty interior; see \cite[Chapter~4]{Schneider1993}.

\begin{definition}
A set $C\subseteq \Rd$ is said to be $p$-convex, if it has the form
\[
C=\bigcup_{i=1}^{p} \overline C_i\,,
\]
where $\overline C_1,\ldots,\overline C_{p}$ are convex bodies in $\Rd$. 
\end{definition}

The geometric requirement on the discrete index sets $(D_n)$ is in fact given in terms of requirements on an associated $p$-convex set, which is assumed to exist. 
%To be more precise, 
Essentially we require that the so-called intrinsic volumes of a down-scaled version of the convex bodies defining the $p$-convex set are bounded. If the $p$-convex set increases similarly in 
all directions, this simply means that the intrinsic volumes has to be comparable in order to certain powers of the volume of the set. The $j$th intrinsic volume $V_j(C)$ (for $j=1,\ldots,d$) describes the geometry of the convex body $C$, and additionally, the collection of intrinsic volumes constitutes an essential part in the famous Steiner formula from convex geometry. As indicated in the introduction, the intrinsic volumes have a direct geometrical meaning, among which we mention a few: $V_0(C)=1$, $V_1(C)$ is proportional to the mean width, $V_{d-1}(C)$ is half the surface area, and $V_d(C)=\abs{C}$ is the volume of $C$. Moreover, the intrinsic volumes satisfy some important properties including non-negativity, i.e.  $V_j(C)\geq 0$, homogeneity, i.e. $V_j(\gamma C)=\gamma^j V_j(C)$ for all $\gamma>0$, and monotonicity, i.e. $V_j(C)\leq V_j(D)$ for $C\subseteq D$.

We can now state the sufficient assumption on the index sets. Due to stationarity of the random fields considered in the paper, we can and do without loss of generality assume that $0\in D_n$, even though this is not stated in the assumption. 
Note also that the assumption as such does not mention the intrinsic volumes but rather a surrounding box with volume proportional to the size of the index set. However, as described in the subsequent lemma, this is essentially equivalent to the requirement on the intrinsic volumes presented in \eqref{ass:boundedintvolumes}. The assumption below is therefore slightly more general than that of e.g. \cite{StehrRonnNielsen2020} in which connectivity combined with the intrinsic volume requirements are assumed, and the same order of expansion in all directions is required.

\begin{assumption}\label{ass:Cnassumption}
For the sequence $(D_n)_{n\in\NN}$ of subsets of $\Zd$ there exists a sequence $(C_n)_{n\in\NN}$ of $p$-convex sets such that $D_n = C_n \cap \Zd$, where
\[
	C_n=\bigcup_{i=1}^{p} C_{n,i} 
\]
and $\abs{C_n}\to \infty$ as $n\to\infty$. Moreover, there is a sequence of $d$-dimensional vectors $(\bb{c_n})_{n\in\NN}$ each with elements 
$0<c_{n,1},\dots,c_{n,d}<\infty$ such that all of the following is satisfied:
\begin{enumerate}[label=\normalfont(\roman*)]
\item \label{item:cn1}$c_{n,\ell} \to \infty$ as $n\to\infty$ for all $\ell=1,\dots,d$.
\item \label{item:cn2}$c_{n,1} \cdots c_{n,d}	\sim \abs{C_n}$ as $n\to\infty$.
\item There exists $0<c<\infty$ such that
\begin{equation}\label{eq:geoassumption1}
		C_n \subseteq \bb{c_n} [-c,c]^d
%		\bigtimes_{\ell=1}^d \bigl[-c\, c_{n,\ell}
%		\:,\:
%		c\, c_{n,\ell} \bigr]
\end{equation}
for all $n$.
\end{enumerate}
\end{assumption}

In the remainder of the paper, when considering a sequence of sets satisfying Assumption~\ref{ass:Cnassumption}, we will refer to the vectors $\bb{c_n}$ as scaling vectors.

\begin{lemma}
Let $C_n$ be a union of convex bodies as defined in Assumption~\ref{ass:Cnassumption}, and let $(\bb{c_n})$ be scaling vectors satisfying Assumption~\ref{ass:Cnassumption} \ref{item:cn1}--\ref{item:cn2}. If \eqref{eq:geoassumption1} is satisfied then
\begin{equation}\label{ass:boundedintvolumes}
	\sum_{i=1}^p V_j(\bb{c_n^{-1}}C_{n,i})
	\quad
	\text{is bounded in $n$ for each }j=1,\dots,d-1 .
\end{equation}
If $C_n$ is also connected then \eqref{eq:geoassumption1} and \eqref{ass:boundedintvolumes} are equivalent.
\end{lemma}

\begin{proof}
The first claim of the lemma follows simply be the monotonicity and homogeneity of the intrinsic volumes. The second claim follows from Theorem~\ref{thm:maingeometrictheorem1}\ref{eq:geomthm4new} below.
\end{proof}

The next example illustrates how Assumption~\ref{ass:Cnassumption} looks like when $C_n$ is connected and expands at the same pace in all directions. This assumption was the one used in the papers \cite{StehrRonnNielsen2021a,StehrRonnNielsen2020,StehrRonnNielsen2021b}.

\begin{example}\label{ex:simpleassumption}
Let $D_n=C_n \cap \Zd$, where $C_n = \cup_{i=1}^p C_{n,i}$ is a $p$-convex set for all $n$. If $C_n$ is connected and $\bb{c_n} = (\abs{C_n}^{1/d},\dots,\abs{C_n}^{1/d})$ for all $n$, then Assumption~\ref{ass:Cnassumption} is satisfied with this scaling vector $\sca$ if and only if
\begin{equation*}
	\frac{\sum_{i=1}^p V_j(C_{n,i})}{\abs{C_n}^{j/d}}
	\quad
	\text{is bounded in $n$ for each }j=1,\dots,d-1 .
\end{equation*} 
\end{example}

We approximate the index sets $D_n$ by a union of certain increasing boxes $\Jz$, $z\in\Zd$, where $n\to\infty$ and $k_n\in\NN$ tends to $\infty$ at a sufficiently slow rate given later. The full approximation scheme is given after this paragraph. The parameter $k_n\in\NN$ determines the size of these boxes relative to the size of $D_n$. In particular, as it will be shown in Theorem~\ref{thm:maingeometrictheorem1}\ref{eq:geomthm2new}, we can approximate $D_n$ by a union of approximately $k_n$ such boxes, with the approximation improving for increasing $n$. 
The approximation scheme is very similar to that of \cite{StehrRonnNielsen2020}, although in that paper $k_n=k$ is constant and independent of $n$.

Let $(D_n)_{n\in\NN}$ be a sequence of sets satisfying Assumption~\ref{ass:Cnassumption}, and let $C_n$ be the associated $p$-convex sets, i.e. $D_n = C_n \cap \Zd$ for all $n\in\NN$. Furthermore, let $(\bb{c_n})$ be the sequence of $d$-dimensional scaling vectors appearing in the assumption, each with elements $c_{n,1},\dots,c_{n,d}$. Let $(k_n)$ be a sequence satisfying that $k_n\to\infty$ and $k_n^{1/d} = o(c_{n,\ell})$ for all $\ell=1,\dots,d$. Note that this will be implied by growth assumptions specified later. For each $n\in\NN$ we define the integers
\[
	t_{n,k_n}^\ell = \lfloor  c_{n,\ell}/k_n^{1/d}\rfloor,
	\qquad\text{for all }\ell=1,\dots,d ,
\]
and we collect these integers in the vector $\bb{t_{n,k_n}}$. 
This vector notation will be used throughout the remainder of the paper.
For each $z=(z_1,\ldots,z_d)\in\ZZ^d$ we now define $I^{n,k_n}_z$ to be the box with corner $\bb{t_{n,k_n}}z$ and side-lengths given by the entries of $\bb{t_{n,k_n}}$, i.e.
\[
	I_z^{n,k_n}= \bb{t_{n,k_n}} \bigl(z + [0,1)^d \bigr)=
	\bigtimes_{\ell=1}^d\big[z_\ell t_{n,k_n}^\ell,(z_\ell+1)t_{n,k_n}^\ell\big)\,.
\]
The idea of the sets $I^{n,k_n}_z$ is that they can be used to approximate the
$C_n$-sets better and better by increasing $n$. For this reason, let
$P_{n,k_n}$ be the set of indices $z$ for which $I_z^{n,k_n}$ is contained in $C_n$, and let $Q_{n,k_n}$ be the set of indices $z$ for which $I_z^{n,k_n}$ is intersected by $C_n$. That is,
\[
	P_{n,k_n}=\{z\in\ZZ^d\::\: I_z^{n,k_n}\subseteq C_n\}\,,\quad\text{and}\quad 
	Q_{n,k_n}=\{z\in\ZZ^d\::\: I_z^{n,k_n}\cap C_n\neq \emptyset\}\,.
\]
Moreover, we let $p_{n,k_n}=\abs{P_{n,k_n}}$ and $q_{n,k_n}=\abs{Q_{n,k_n}}$ be the size of those sets and note that, by construction, $\limsup_{n\to\infty }p_{n,k_n}/k_n \le 1$ and $\liminf_{n\to\infty }q_{n,k_n}/k_n\ge 1$. 

To approximate the index sets $D_n$ we use the lattice points $J_z^{n,k_n}=I_z^{n,k_n}\cap \ZZ^d$ of $I_z^{n,k_n}$, and define
\begin{equation}\label{eq:plusandminussets1}
	\dnm=\bigcup_{z\in P_{n,k_n}}J_z^{n,k_n}
	\qquad\text{and}\qquad 
	\dnp=\bigcup_{z\in Q_{n,k_n}}J_z^{n,k_n} .
\end{equation}
Since $J_z^{n,k_n}\subseteq D_n$ for all $z\in P_{n,k_n}$, and $z\in Q_{n,k_n}$ for all $J_z^{n,k_n}$ with $J_z^{n,k_n}\cap D_n\neq\emptyset$, we have the approximation
\begin{equation}\label{fml:Cnplusminus1}
	\dnm
	\subseteq D_n
	\subseteq \dnp .
\end{equation} 

As mentioned, the following theorem gives the quality of the approximation scheme, when approximating $D_n$ using the sets $J_z^{n,k_n}$. The proof can be found in Section~\ref{appnA1} in Appendix~\ref{appnA}.

\begin{theorem}\label{thm:maingeometrictheorem1}
Let $(D_n)_{n\in\NN}$ satisfy Assumption~\ref{ass:Cnassumption}, and let $C_n$ be the $p$-convex  set associated with $D_n$ by $D_n=C_n\cap \ZZ^d$. Furthermore, let $(\bb{c_n})$ be the sequence of $d$-dimensional scaling vectors each with elements $0<c_{n,1},\dots,c_{n,d}<\infty$ introduced in Assumption~\ref{ass:Cnassumption}, and let $k_n\to\infty$ be a sequence satisfying $k_n^{1/d}=o(c_{n,\ell})$ for each $\ell$. 
Then
\begin{enumerate}[label=\normalfont(\roman*)]
	\item \label{eq:geomthm1new} $\abs{D_n}\sim \abs{C_n}$ as $n\to\infty$,
	\item \label{eq:geomthm2new} the sequences $p_{n,k_n}$ and $q_{n,k_n}$, defined above, satisfy that
	\[
		\lim_{n\to\infty}\frac{p_{n,k_n}}{k_n}=
		\lim_{n\to\infty}\frac{q_{n,k_n}}{k_n}=1 ,
	\]
\item\label{eq:geomthm4new} there is a set $K_n$, that can be chosen independently of $(k_n)_{n\in\NN}$, defined by
	\begin{align*}
	K_n & = \Zd \cap \bb{c_n} [-c,c]^d \\ &=
	\Zd \cap 
\bigtimes_{\ell=1}^d \bigl[-c\cdot c_{n,\ell}
		\:,\:
		c\cdot c_{n,\ell} \bigr] 
	\end{align*}
	for some $0<c<\infty$, such that $\dnp \subseteq K_n$
	for all $n\in\NN$. 
\end{enumerate}
If the sets $(C_n)$ are all connected and satisfy \eqref{ass:boundedintvolumes} instead of \eqref{eq:geoassumption1}, then Assumption~\ref{ass:Cnassumption} still holds and in particular \ref{eq:geomthm1new}--\ref{eq:geomthm4new} does so too.
\end{theorem}

%----------------------------
	\section{Behavior of the maximum}
	\label{sec:independenceapproximation}
%----------------------------

In this paper we consider a stationary random field $(\xi_v)_{v\in \Zd}$ and a sequence of index sets $(D_n)_{n\in \NN}$ with $D_n \subseteq\Zd$ satisfying Assumption~\ref{ass:Cnassumption} above. Under the assumption of certain mixing and local conditions, we present results relating the distribution of the tail of the individual variables to the distribution of the maximum of the field over $D_n$. 
As mentioned in the previous section, the geometric approximation used in this paper is to some extend a generalization of that of \cite{StehrRonnNielsen2020} in that the approximation is constructed with $n$-dependent $k_n$. Moreover, the conditions and associated results of this section naturally generalize those of \cite{StehrRonnNielsen2020} by allowing index sets to increase with a different pace in different directions and, perhaps more importantly, a certain degree of clustering of extremes.

Before presenting conditions and results, we introduce some relevant notation used throughout the paper. For a $d$-dimensional vector $\bb{\gamma}$, we say that two subsets $A,B$ of $\Zd$ are $\bb{\gamma}$-separated if $B \subseteq (A \oplus\bb{\gamma}B(1))^c$. If $\bb{\gamma}=(\gamma,\dots,\gamma)$ for some $0<\gamma<\infty$, we say that the sets are $\gamma$-separated.
Moreover, for two vectors, as for instance $\bb{\gamma_n}=(\gamma_{n,1},\dots,\gamma_{n,d})$ and $\bb{c_n}=(c_{n,1},\dots,c_{n,d})$ used below, we write $\bb{\gamma_n} = o(\bb{c_n})$ if $\gamma_{n,\ell}/c_{n,\ell}\to 0$ for each $\ell=1,\dots,d$. Lastly, for any $A\subseteq \Zd$ we shorten $\max_{v\in A} \xi_v$ by $M_\xi(A)$.

In the remainder of the paper we consider index sets $(D_n)_{n\in\NN}$ satisfying Assumption~\ref{ass:Cnassumption}. In particular, Theorem~\ref{thm:maingeometrictheorem1} holds. In fact, the mixing condition $\mathcal{D}(x_n;K_n)$ given shortly is described in terms of the surrounding box $K_n \supseteq D_{n,k_n}^+$ that exists due to Theorem~\ref{thm:maingeometrictheorem1}\ref{eq:geomthm4new}. 

\begin{cond1}
\label{cond:mixingcond}
The condition $\mathcal{D}(x_n;K_n)$ is satisfied for the stationary field $(\xi_v)_\vZ$ if there exists an increasing sequence $(\bb{\gamma_n})$ of $d$-dimensional vectors such that: 1) $\bb{\gamma_n}=o(\bb{c_n})$ as $n\to\infty$ and 2) for all $n\in \NN$ and all $\bb{\gamma_n}$-separated sets $A,B \subseteq K_n$
where at least one is a box, it holds that
\begin{equation}\label{eq:mixingcond}
	\abs[\big]{
	\PP (M_\xi(A \cup B) \le x_n ) -
	\PP (M_\xi(A) \le x_n )
	\PP (M_\xi(B) \le x_n )
	}
	\le \alpha_n,
\end{equation}
where $\alpha_n\to 0$ as $n\to \infty$.
\end{cond1}

\begin{example}\label{ex:mdependent}
Suppose $(\xi_v)_{v\in\Zd}$ is an $m$-dependent random field for some $m\in\NN$: For all $m$-separated sets $A,B\subseteq\Zd$ it holds that $(\xi_v)_{v\in A}$ and $(\xi_v)_{v\in B}$ are independent. If furthermore $(D_n)_{n\in\NN}$ is a sequence of sets satisfying Assumption~\ref{ass:Cnassumption}, then for any sequence $(x_n)$ the condition $\mathcal{D}(x_n;K_n)$ is satisfied. The sequence $(\bb{\gamma_n})$ should satisfy $\bb{\gamma_n}=o(\bb{c_n})$ as $n\to\infty$ and $\gamma_{n,\ell}\geq m$ eventually for each $\ell=1,\ldots,d$. Note that $(\alpha_n)$ can be chosen to be 0 eventually.
\end{example}

If $(\xi_v)_\vZ$ is an i.i.d. random field, it is well known that
\begin{equation}\label{eq:iidconv}
	\abs{D_n} \PP(\xi_0>x_n) \to \tau
	\quad \text{if and only if}\quad
	\PP(M_\xi(D_n) \le x_n) \to \exp(-\tau)
\end{equation}
as $n\to\infty$; see e.g. \cite[Theorem~1.5.1.]{Leadbetter1983}.
However, as the following example illustrates, the approximate independence implied by the mixing condition $\mathcal D$ is not enough to guarantee the equivalence \eqref{eq:iidconv} to hold true for dependent fields. The example is a spatial generalization of the classical example of a one-dimensional process with local maximum occurring in clusters of a given size; see e.g. \cite[Example~4.4.2]{Embrecths1997}.

\begin{example}\label{ex:clusterexample}
Let $(D_n)_{n\in\NN}$ be a sequence of sets satisfying Assumption~\ref{ass:Cnassumption}, and let $F$ be a distribution function chosen such that 
\[
	\abs{D_n} \FF(x_n) = \abs{D_n}(1- F(x_n)) \to \tau
\]
as $n\to\infty$ for some sequence $(x_n)$ and some $\tau\in(0,\infty)$. Let $B \subseteq \Zd$ be a finite set with $\abs{B}\ge 1$, and let $(Y_z)_{z\in\Zd}$ be a field of i.i.d. random variables with common distribution function $F^{1/\abs{B}}$. Define the stationary field $(\xi_v)_\vZ$ by
\[
	\xi_v = \max_{z\in v + B} Y_z ,
\]
and note that $\xi_v$ has distribution function $F$ and that $(\xi_v)_{v\in\Zd}$ is $m$-dependent for $m$ large relative to the span of $B$. Thus condition $\mathcal{D}(x_n;K_n)$ is satisfied, cf. Example~\ref{ex:mdependent}. From the convergence of the tail of $F$ it is seen that
\begin{equation}\label{eq:example1}
	\abs{D_n} \PP(Y_0 > x_n) \to \tau/\abs{B} ,
\end{equation}
which will be used shortly.
Appealing to Lemma~\ref{lem:applemma1} in Appendix~\ref{appnA} we have that the number of points $\abs{D_n \oplus B}$ in the Minkowski sum $D_n \oplus B$ is asymptotically equivalent to $\abs{D_n}$. For $n\to\infty$, we therefore conclude that
\begin{align*}
	\PP \bigl(\max_{v\in D_n} \xi_v \le x_n \bigr) &
	= \PP \bigl(\max_{z\in D_n \oplus B} Y_v \le x_n \bigr) \\&
	= \PP\bigl( Y_0 \le x_n \bigr)^{\abs{D_n \oplus B}}\\&
	\sim \PP\bigl( Y_0 \le x_n \bigr)^{\abs{D_n}} \\&
	\to \exp\bigl(-\tau/\abs{B} \bigr) ,
\end{align*}
where the asymptotic equivalence and the convergence follow from \eqref{eq:example1} by standard arguments.
\end{example}

In fact, under the assumption of the mixing condition above, the distribution function of the maximum $M_\xi(D_n)$ is close to the $k_n$th power of the distribution function of the maximum taken over certain $\bb{\gamma_n}$-separated boxes. 
More precisely, the lemma below holds with subsets $H_z^{n,k_n}$ of $J_z^{n,k_n}$ defined by
\[
	H_z^{n,k_n}
	= \bigl\{
		u\in\Zd\::\:
		z_\ell t_{n,k_n}^\ell \le u_\ell \le 
		(z_\ell+1) t_{n,k_n}^\ell - 1 - \gamma_{n,\ell},\ 
		\text{for all }\ell=1,\dots,d
	\bigr\} .
\]
Note that they are indeed $\bb{\gamma_n}$-separated for varying $z\in\Zd$.

\begin{lemma}\label{lem:Happrox}
Let $(D_n)_{n\in\NN}$ be a sequence of sets satisfying Assumption~\ref{ass:Cnassumption}, and let $(\xi_v)_\vZ$ be a stationary field satisfying $\mathcal{D}(x_n;K_n)$ for some sequence $(x_n)_{n\in\NN}$. If $k_n$ is a sequence of integers satisfying $k_n\to\infty$, $k_n \alpha_n \to 0$ and $k_n^{1/d} \bb{\gamma_n} = o(\bb{c_n})$ as $n\to\infty$, then
\begin{equation}\label{eq:knapprox0}
	\PP\bigl(M_\xi(D_n) \le x_n\bigr) \le 
	\PP^{k_n}\bigl(M_\xi(H_0^{n,k_n}) \le x_n\bigr) + o(1)
\end{equation}
as $n\to\infty$. If furthermore
\begin{equation}\label{eq:tailcondition0}
	\limsup_{n\to\infty} \abs{D_n} \PP(\xi_0>x_n) < \infty ,
\end{equation}
then
\begin{equation}\label{eq:knapprox}
\begin{aligned}
	\PP\bigl(M_\xi(D_n) \le x_n\bigr) & 
	= \PP^{k_n}\bigl(M_\xi(H_0^{n,k_n}) \le x_n\bigr) + o(1) \\ & 
	= \PP^{k_n}\bigl(M_\xi(J_0^{n,k_n}) \le x_n\bigr) + o(1)
\end{aligned}
\end{equation}
as $n\to\infty$.
\end{lemma}

Before proving the lemma, we give a brief remark on the equality \eqref{eq:knapprox}. In the one-dimensional case, it is stated in \cite{Leadbetter1983a} that the equality holds even without the tail assumption \eqref{eq:tailcondition0}. This is however not the case unless stronger assumptions on the increase rate of $k_n$ are made.

\begin{proof}
Using the set $\dnm$ introduced in \eqref{eq:plusandminussets1} and the fact that $H_z^{n,k_n} \subseteq J_z^{n,k_n}$, we find that
\begin{align*}
	\PP\bigl( M_\xi(\dnm) \le x_n \bigr) &
	\le \PP \Bigl( \bigcap_{z\in P_{n,k_n}} \{ M_\xi(H_z^{n,k_n}) \le x_n \} \Bigr) \\&
	\le \PP^{p_{n,k_n}} \bigl( M_\xi(H_0^{n,k_n}) \le x_n \bigr) + (p_{n,k_n}-1) \alpha_n ,
\end{align*}
where the last inequality follows from the mixing condition \eqref{eq:mixingcond} by induction. 
Since $p_{n,k_n}\sim k_n$ as shown in Theorem~\ref{thm:maingeometrictheorem1}\ref{eq:geomthm2new}, and $k_n\alpha_n \to 0$ by assumption, the inequality \eqref{eq:knapprox0} follows as 
\[
	\PP \bigl( M_\xi(D_{n}) \le x_n \bigr) 
	\le \PP\bigl(M_\xi(\dnm) \le x_n \bigr) .
\]

Now assume that \eqref{eq:tailcondition0} is satisfied, and let $H_0^{*}=J_0^{n,k_n} \setminus H_0^{n,k_n}$
(omitting the $n,k_n$ dependence of $H_0^*$). 
Note that, up to a constant, 
\[
	\abs{H_0^{*}} \sim 
	\Bigl(\sum_{\ell=1}^d \frac{\gamma_{n,\ell}}{t_{n,k_n}^\ell} \Bigr)
	\prod_{\ell=1}^d t_{n,k_n}^\ell  .
\]
Following the lines of \cite[Lemma~3]{StehrRonnNielsen2020} (which is under the set expansion given in Example~\ref{ex:simpleassumption} though) and utilizing that $\dnm \subseteq D_{n} \subseteq \dnp$, $p_{n,k_n} \sim q_{n,k_n}\sim k_n$ and $k_n\alpha_n \to 0$ yield
\begin{align*}
	\MoveEqLeft[5]	
	\PP \bigl( M_\xi(D_{n}) \le x_n \bigr)
	= \PP^{k_n} \bigl( M_\xi(J_0^{n,k_n}) \le x_n \bigr) \\ &
	+
	k_n \PP\bigl( M_\xi(H_0^{n,k_n}) \le x_n < M_\xi(H_0^*)\bigr) + o(1)
\end{align*}
as $n\to\infty$. 
Since each $t_{n,k_n}^\ell \sim c_{n,\ell}/k_n^{1/d}$ giving in particular that $\prod_\ell t_{n,k_n}^\ell \sim \abs{D_n}/k_n$, we obtain
\begin{align*}
	k_n \PP\bigl( M_\xi(H_0^{n,k_n}) \le x_n < M_\xi(H_0^*)\bigr) &
	\le k_n \PP\bigl( M_\xi(H_0^*)>x_n\bigr) \\ &
	\le k_n \abs{H_0^*} \PP(\xi_0 >x_n) \\ &
	\sim \Bigl(\sum_{\ell=1}^d \frac{k_n^{1/d}\gamma_{n,\ell}}{c_{n,\ell}} \Bigr) \abs{D_n} \PP(\xi_0 >x_n)
\end{align*}
as $n\to\infty$. By the tail assumption \eqref{eq:tailcondition0} and the assumption that $k_n^{1/d} \bb{\gamma_n} = o(\bb{c_n})$ we conclude that
\begin{equation}\label{eq:helpequation1}
	k_n \PP\bigl( M_\xi(H_0^{n,k_n}) \le x_n < M_\xi(H_0^*)\bigr) \to 0
\end{equation}
as $n\to\infty$. This gives one of the equalities in \eqref{eq:knapprox}. The fact that $J_0^{n,k_n}$ can be substituted by $H_0^{n,k_n}$ in \eqref{eq:knapprox} is a consequence of \eqref{eq:helpequation1} since
\begin{align*}
	0 & \le 
	\PP^{k_n} \bigl( M_\xi(H_0^{n,k_n}) \le x_n \bigr) -
	\PP^{k_n} \bigl( M_\xi(J_0^{n,k_n}) \le x_n \bigr) \\ &
	\le k_n \PP\bigl( M_\xi(H_0^{n,k_n}) \le x_n < M_\xi(H_0^*)\bigr).
\end{align*}
This concludes the proof.
\end{proof}

Let $\preceq$ be an arbitrary translation invariant total order on $\Zd$ and define
\[
	A_{v}^{(m)}
	= \{ z \in v + [-m,m]^d \cap \Zd) 
	\::\: v \prec z
	\} 
\]
for an integer $m \in\NN$ and $\vZ$. Similarly, we define
\[
	A_{v}^{n,k_n}
	= \{ z \in v + \bb{t_{n,k_n}}\,[-1,1]^d \cap \Zd) 
	\::\: v \prec z
	\} ,
\]
which has varying side-length given by the $d$-dimensional vector $\bb{t_{n,k_n}}$.

\begin{lemma}\label{lem:maxrelation}
Let $(D_n)_{n\in\NN}$ be a sequence of sets satisfying Assumption~\ref{ass:Cnassumption}, and let $(\xi_v)_\vZ$ be a stationary field satisfying $\mathcal{D}(x_n;K_n)$ for some sequence $(x_n)_{n\in\NN}$. If $k_n$ is a sequence of integers satisfying $k_n\to\infty$, $k_n \alpha_n \to 0$ and $k_n^{1/d} \bb{\gamma_n} = o(\bb{c_n})$ as $n\to\infty$, then
\begin{equation}\label{eq:knineq1}
	\PP \bigl(M_\xi(D_n)\le x_n \bigr) 
	\le 
	\exp \Bigl( - \abs{D_n} 
	\PP\bigl(M_\xi(A_0^{n,k_n})\le x_n < \xi_0 \bigr) \Bigr) + o(1)
\end{equation}
as $n\to\infty$. If furthermore \eqref{eq:tailcondition0} is satisfied then
\begin{equation}\label{eq:knequality}
	\PP \bigl(M_\xi(D_n)\le x_n \bigr) 
	=
	\exp \Bigl( - \abs{D_n} 
	\PP\bigl(M_\xi(A_0^{n,k_n})\le x_n < \xi_0 \bigr) \Bigr) + o(1)
\end{equation}
as $n\to\infty$.
\end{lemma}

\begin{proof}
Using the fact that
\begin{equation}\label{eq:Obrien}
	a_n^n - \exp(-n(1-a_n)) \to 0,
	\qquad a_n\in [0,1],
\end{equation}
see \cite[Formula~(2.8)]{Obrien1987}, it follows from \eqref{eq:knapprox0} that
\begin{equation}\label{eq:knapprox01}
	\PP\bigl(M_\xi(D_n) \le x_n\bigr) 
	\le 
	\exp\Bigl(
	- k_n	
	\PP\bigl(M_\xi(H_0^{n,k_n}) > x_n\bigr) \Bigr) + o(1) .
\end{equation}
As the order $\preceq$ is translation invariant we obtain the inequality
\begin{align*}
	\PP\bigl(M_\xi(H_0^{n,k_n}) > x_n\bigr) & 
	= \sum_{v\in H_0^{n,k_n}} \PP \bigl( \, \max_{v \prec z \in H_0^{n,k_n}} \xi_z \le x_n < \xi_v \bigr) \\ &
	\ge \sum_{v\in H_0^{n,k_n}} \PP \bigl(M_\xi(A_v^{n,k_n}) \le x_n < \xi_v \bigr) \\ &
	= \abs{H_0^{n,k_n}} \PP \bigl(M_\xi(A_0^{n,k_n}) \le x_n < \xi_0 \bigr) .
\end{align*}
Since $k_n^{1/d} \bb{\gamma_n} = o(\bb{c_n})$, which gives that $\abs{H_0^{n,k_n}} \sim \abs{D_n}/k_n$, formula \eqref{eq:knapprox01} now implies \eqref{eq:knineq1}.

For the opposite inequality assume that \eqref{eq:tailcondition0} is satisfied. Then, due to \eqref{eq:knapprox} and \eqref{eq:Obrien},
\begin{equation}\label{eq:limitexpequality}
	\PP\bigl(M_\xi(D_n)\le x_n \bigr) = 
	\exp\Bigl(-k_n \PP\bigl(M_\xi(J_0^{n,k_n})> x_n \bigr)
	\Bigr) + o(1) . 
\end{equation}
Since also $\abs{J_0^{n,k_n}} \sim \abs{D_n}/k_n$ we obtain that
\begin{equation*}%\label{eq:liminfcondition}
	\liminf_{n\to\infty}
	\PP\bigl(M_\xi(D_n)\le x_n \bigr) > 0 .
\end{equation*}
Utilizing this, the same arguments used in the latter part of the proof of \cite[Theorem~3.1.]{Soja2019} (more precisely those leading to the reverse of their inequality ($3.4$)) show that the reverse inequality to \eqref{eq:knineq1} also holds, and thus we obtain equality as in \eqref{eq:knequality}.
\end{proof}

With the lemma above in mind, it is natural to consider the local condition $\mathcal D^\ell$ below. The condition, that introduces potential clustering of exceedances of $x_n$, is defined for $\theta \in [0,1]$.

\begin{cond2}% [$\mathcal{D}^\ell(x_n;k_n; \theta)$]
The condition $\mathcal{D}^\ell(x_n;k_n; \theta)$ is satisfied for the stationary field $(\xi_v)_\vZ$ if 
\begin{equation}\label{eq:thetacondition}
	\PP \bigl(
		M_\xi(A_0^{n,k_n})\le x_n \mid \xi_0>x_n
	\bigr)
	\to \theta
\end{equation}
as $n\to\infty$.
\end{cond2}

The following result is a generalization of \cite[Theorem~5]{StehrRonnNielsen2020} allowing clustering of extremes in the sense of \eqref{eq:thetacondition}, and allowing the more general expansion of $(D_n)_{n\in\NN}$ as assumed in Assumption~\ref{ass:Cnassumption}.

\begin{theorem}\label{thm:maximumtheorem}
Let $(D_n)_{n\in\NN}$ be a sequence of sets satisfying Assumption~\ref{ass:Cnassumption}, and let $(\xi_v)_\vZ$ be a stationary field satisfying $\mathcal{D}(x_n;K_n)$ for some sequence $(x_n)_{n\in\NN}$. Moreover, let $k_n$ be a sequence of integers satisfying $k_n\to\infty$, $k_n \alpha_n \to 0$ and $k_n^{1/d} \bb{\gamma_n} = o(\bb{c_n})$ as $n\to\infty$, such that $\mathcal{D}^\ell(x_n;k_n; \theta)$ is satisfied for some $\theta\in (0,1]$. Then, for all $ 0 \le \tau < \infty$,
\[
	\abs{D_n} \PP(\xi_0>x_n) \to \tau
\]
if and only if
\[
	\PP\bigl(\max_{v\in D_n} \xi_v \le x_n\bigr) \to \exp(-\theta \tau) 
\]
as $n\to\infty$.
\end{theorem}

\begin{proof}
Assume first that $\abs{D_n} \PP(\xi_0>x_n) \to \tau$ as $n\to\infty$ for some finite $\tau$. In particular, condition \eqref{eq:tailcondition0} is satisfied and hence also \eqref{eq:knequality} is so. Using \eqref{eq:thetacondition} it now follows easily that
\[
	\lim_{n\to\infty} \PP \bigl(M_\xi(D_n)\le x_n \bigr)
	= \exp( - \theta  \tau) .
\]

Assume conversely that
$	\lim_{n\to\infty} \PP(M_\xi(D_n)\le x_n)
	= \exp( - \theta  \tau)
$
for some finite $\tau$. By an application of \eqref{eq:knineq1} this in particular implies that
\[
	\limsup_{n\to\infty} \abs{D_n}
	\PP\bigl(M_\xi(A_0^{n,k_n})\le x_n < \xi_0 \bigr) < \infty .
\]
Since $\theta >0$ we obtain that \eqref{eq:tailcondition0} is satisfied, and thus the equality \eqref{eq:knequality} holds. Hence
\begin{align*}
	\lim_{n\to\infty} \abs{D_n} \PP(\xi_0>x_n) &
	= \lim_{n\to\infty}\frac{\log \exp\Bigl(
	- \abs{D_n} \PP\bigl(M_\xi(A_0^{n,k_n})\le x_n < \xi_0 \bigr)
	\Bigr)}{\PP \bigl(
		M_\xi(A_0^{n,k_n})\le x_n \mid \xi_0>x_n
	\bigr)} \\ &
	= \frac{\log \exp(-\theta \tau)}{\theta}
	= \tau 
\end{align*}
completing the proof.
\end{proof}

We note that a range of anti-clustering conditions are found in the literature. Below we include a condition from \cite{Soja2019} which is a multi-dimensional counterpart of the condition $D^{(m+1)}(x_n)$ from \cite{Chernick1991} and which is satisfied by e.g. $m$-dependent fields under the assumption of \eqref{eq:tailcondition0}; see \cite[Section~5.1]{Soja2019}.

\begin{cond3}% [$\mathcal{D}^{(m)}(x_n;k_n)$]
The condition $\mathcal{D}^{(m)}(x_n;k_n)$ is satisfied for the stationary field $(\xi_v)_\vZ$ if 
\begin{equation}\label{eq:anti-clustering0}
	\abs{D_n} \, \PP \Bigl(
	M_\xi(A_0^{(m)})\le x_n < \xi_0\,,\,
	M_\xi(A_0^{n,k_n}\setminus A_0^{(m)}) > x_n
	\Bigr)
	\to 0
\end{equation}
as $n\to\infty$, with the convention that $M_\xi(A_0^{(0)})=M_\xi(\emptyset) = -\infty$.
\end{cond3}

Obviously, the condition \eqref{eq:anti-clustering0} is satisfied if
\begin{equation}\label{eq:anti-clustering}
	\abs{D_n} \sum_{v\in A_0^{n,k_n}\setminus A_0^{(m)}} 
	\PP \Bigl(
	M_\xi(A_0^{(m)})\le x_n < \xi_0\,,\,
	\xi_v > x_n
	\Bigr)
	\to 0 .
\end{equation}
In fact, setting $m=0$ in \eqref{eq:anti-clustering} corresponds to the anti-clustering condition $\mathcal{D}'(x_n)$ of \cite{StehrRonnNielsen2020} (which however, was formulated slightly incorrect) now with $n$-dependent $k_n$.

Clearly, under the assumption of the anti-clustering condition $\mathcal D^{(m)}(x_n;k_n)$, the local condition $\mathcal D^\ell$ is equivalent to the much simpler convergence
\begin{equation}\label{eq:thetacondition2}
	\lim_{n\to\infty }\PP \bigl( M_\xi(A_0^{(m)})\le x_n \mid \xi_0>x_n \bigr)
	= \theta ,
\end{equation}
at least if $\liminf_{n\to \infty} \abs{D_n} \PP(\xi_0>x_n)>0$. The following lemma, which follows simply by rearranging the probabilities, gives this and another relation between the conditions which are used throughout the paper.

\begin{lemma}\label{lem:conditionrelation}
Assume that $\liminf_{n\to \infty} \abs{D_n} \PP(\xi_0>x_n)>0$. Then the relations below hold:
\begin{enumerate}[label=\normalfont(\roman*)]
	\item \label{eq:conditionrelation1}
		If $\mathcal D^{(m)}(x_n;k_n)$ is satisfied for some $m\in\NN_0$, then $\mathcal{D}^\ell(x_n;k_n;\theta)$ is equivalent to \eqref{eq:thetacondition2} for this $m$.
	\item \label{eq:conditionrelation2}
		Assume that $\limsup_{n\to \infty} \abs{D_n} \PP(\xi_0>x_n)<\infty$. If $\mathcal D^{\ell}(x_n;k_n; 1)$ is satisfied, then $\mathcal D^{(m)}(x_n;k_n)$ and \eqref{eq:thetacondition2} hold for $\theta=1$ and all $m\in\NN_0$. In particular, $\mathcal D^{\ell}(x_n;k_n; 1)$ is satisfied 
		if and only if $\mathcal D^{(m)}(x_n;k_n)$ and \eqref{eq:thetacondition2} hold for $\theta=1$ and some $m\in\NN_0$, which is satisfied
		if and only if $\mathcal D^{(m)}(x_n;k_n)$ and \eqref{eq:thetacondition2} hold for $\theta=1$ and all $m\in\NN_0$.
\end{enumerate}
\end{lemma}

The corollary to Theorem~\ref{thm:maximumtheorem} given below does in fact not follow directly by an application of Lemma~\ref{lem:conditionrelation}, however, it follows by almost identical arguments as Theorem~\ref{thm:maximumtheorem}.

\begin{corollary}
Let $(D_n)_{n\in\NN}$ be a sequence of sets satisfying Assumption~\ref{ass:Cnassumption}, and let $(\xi_v)_\vZ$ be a stationary field satisfying $\mathcal{D}(x_n;K_n)$ for some sequence $(x_n)_{n\in\NN}$. Moreover, let $k_n$ be a sequence of integers satisfying $k_n\to\infty$, $k_n \alpha_n \to 0$ and $k_n^{1/d} \bb{\gamma_n} = o(\bb{c_n})$ as $n\to\infty$, such that $\mathcal{D}^{(m)}(x_n;k_n)$ is satisfied for some $m\in\NN_0$. Assume furthermore that \eqref{eq:thetacondition2} is satisfied for such $m$ and for some $\theta\in (0,1]$.
Then, for all $ 0 \le \tau < \infty$,
\[
	\abs{D_n} \PP(\xi_0>x_n) \to \tau
\]
if and only if
\[
	\PP\bigl( \max_{v\in D_n} \xi_v \le x_n\bigr) \to \exp(-\theta \tau)
\]
as $n\to\infty$.
\end{corollary}

\begin{remark}
Note that a field satisfying the conditions of the Corollary with $m=0$, but with $\mathcal D^{(m)}(x_n;k_n)$ replaced by the stronger requirement \eqref{eq:anti-clustering}, necessarily satisfies \eqref{eq:thetacondition2} with $\theta=1$. Thus, in this case, the result simplifies to Theorem~5 of \cite{StehrRonnNielsen2020}.
\end{remark}

\setcounter{example}{2}
\begin{example}[Continued]
Recall the setup of Example~\ref{ex:clusterexample}. We concluded that one implication of the corollary was true with $\theta=1/\abs{B}$. As we will see now, under the initial assumption that $(x_n)$ is chosen such that $\limsup_{n\to\infty}\abs{D_n}\PP(\xi_0>x_n)<\infty$, the reverse implication is in fact also true since the anti-clustering condition $\mathcal D^{(m)}(x_n;k_n)$ is satisfied for sufficiently large $m$, and since \eqref{eq:thetacondition2} is satisfied with $\theta=1/\abs{B}$: As already mentioned, $(\xi_v)_\vZ$ is $m$-dependent for $m$ sufficiently large.
Therefore, $D^{(m)}(x_n;k_n)$ is satisfied for any sequence $(k_n)$ with the proper growth rates; see the comment prior to condition $\mathcal D^{(m)}(x_n;k_n)$. Now fix such an $m$. Then,
\begin{align*}
	\MoveEqLeft[1]
	\PP \bigl( M_\xi(A_0^{(m)}) \le x_n < \xi_0 \bigr) \\ &
	= \PP \bigl( M_\xi(A_0^{(m)}) \le x_n \bigr)
	- \PP \bigl( M_\xi(A_0^{(m)}) \le x_n, \xi_0\le x_n \bigr) \\ &
	= \PP \bigl( Y_z \le x_n \text{ for all } z \in A_0^{(m)} \oplus B \bigr)
	- \PP \bigl( Y_z \le x_n \text{ for all } z \in (A_0^{(m)}\cup \{0\}) \oplus B \bigr) .
\end{align*}
Let the elements of $B$ be ordered as follows (relative to the underlying translation invariant order $\preceq$):
\[
	b_1 \prec b_2 \prec \cdots \prec b_{\abs{B}} .
\]
Then, utilizing the fact that $b_j-b_i\in A_0^{(m)}$ for all $i<j$ due to the choice of $m$, we see that
\[
	(A_0^{(m)}\cup \{0\}) \oplus B = (A_0^{(m)} \oplus B) \cup \{b_1\} ,
\]
with $b_1 \not\in A_0^{(m)} \oplus B$.
Consequently, and using the i.i.d. structure of the random variables $(Y_z)_{z\in\Zd}$, the probability above simplifies to
\[
	\PP \bigl( M_\xi(A_0^{(m)}) \le x_n < \xi_0 \bigr)
	= \PP(Y_0 \le x_n)^{\abs{A_0^{(m)} \oplus B}} \,
	\PP(Y_0>x_n) .
\]
Since $A_0^{(m)}\oplus B$ is a finite set independent of $n$, the first factor of the product converges to $1$ as $n\to\infty$. Realizing secondly that $\PP(Y_0>x_n)$ is asymptotically equivalent to $\PP(\xi_0>x_n)/\abs{B}$ shows that \eqref{eq:thetacondition2} is satisfied with $\theta=1/\abs{B}$.
\end{example}

The results below, which follow directly from results and considerations above, relates the existence of an extremal index to the local condition $\mathcal{D}^\ell$: In accordance with the definition given in e.g. \cite{Leadbetter1983}, we say that the stationary field $(\xi_v)_\vZ$ has extremal index $\theta \in [0,1]$ (with respect to $(D_n)$) if for each $\tau >0$ there exists a sequence $x_n(\tau)$ such that
\[
	\abs{D_n} \PP(\xi_0 > x_n(\tau)) \to \tau
\]
and
\[
	\PP \bigl( M_\xi(D_n) \le x_n(\tau) \bigr) \to \exp(-\theta \tau) .
\]
as $n\to\infty$.

\begin{theorem}\label{thm:extremalindextheorem}
Let $(D_n)_{n\in\NN}$ be a sequence of sets satisfying Assumption~\ref{ass:Cnassumption}, and let $(\xi_v)_\vZ$ be a stationary field satisfying $\mathcal{D}(x_n;K_n)$ where for all $0 < \tau<\infty$ the sequence $x_n=x_n(\tau)$ is chosen such that
$\abs{D_n}\PP(\xi_0>x_n(\tau)) \to\tau$. Let $\alpha_n=\alpha_n(\tau)$ and $\gamma_n=\gamma_n(\tau)$ be the mixing constants from the condition. Then the field has extremal index $\theta \in [0,1]$ with respect to $(D_n)$ if and only if there exist a sequence $k_n=k_n(\tau)$ satisfying $k_n\to\infty$, $k_n \alpha_n \to 0$ and $k_n^{1/d} \bb{\gamma_n} = o(\bb{c_n})$ as $n\to\infty$, such that $\mathcal{D}^\ell(x_n(\tau);k_n(\tau);\theta)$ is satisfied for all $\tau>0$.
\end{theorem}

\begin{proof}
The claim follows easily from Lemma~\ref{lem:maxrelation} equation \eqref{eq:knequality}.
\end{proof}

\begin{corollary}
Let the assumptions of Theorem~\ref{thm:extremalindextheorem} be satisfied. Assume furthermore that there exist a sequence $k_n=k_n(\tau)$ satisfying $k_n\to\infty$, $k_n \alpha_n \to 0$ and $k_n^{1/d} \bb{\gamma_n} = o(\bb{c_n})$ as $n\to\infty$, such that the anti-clustering condition $\mathcal{D}^{(m)}(x_n(\tau);k_n(\tau))$ holds for some $m\in\NN_0$ and all $\tau>0$. Then the field has extremal index $\theta \in [0,1]$ with respect to $(D_n)$ if and only if \begin{equation*}
	\lim_{n\to\infty }\PP \bigl( M_\xi(A_0^{(m)})\le x_n(\tau) \mid \xi_0>x_n(\tau) \bigr)
	= \theta
\end{equation*}
for all $\tau>0$.
\end{corollary}

\begin{proof}
Applying Theorem~\ref{thm:extremalindextheorem} and Lemma~\ref{lem:conditionrelation}\ref{eq:conditionrelation1} gives the result.
\end{proof}

\setcounter{example}{2}
\begin{example}[Continued]
The field $(\xi_v)_\vZ$ with $\xi_v=\max_{z\in v+B} Y_z$ introduced in Example~\ref{ex:clusterexample} has extremal index $1/\abs{B}$ with respect to $(D_n)$. 
\end{example}

%----------------------------
 	\section{Cluster counting processes}
	\label{sec:clusterpointmeasures}
%----------------------------	

In the remainder of the paper we focus on sets $D_n$ of the specific structure defined below. Moreover, we only consider choices of the corresponding sequence $x_n$ such that $\abs{D_n}\PP(\xi_0>x_n)$ converges to a finite, non-zero limit. In particular, this means that the local condition $\mathcal{D}^\ell(x_n;k_n;\theta)$ is implied by the anti-clustering condition $\mathcal{D}^{(m)}(x_n;k_n)$ in combination with \eqref{eq:thetacondition2} for some $\theta\in[0,1]$; see Lemma~\ref{lem:conditionrelation}\ref{eq:conditionrelation1}. Though only formulated under the assumption of $\mathcal D^\ell$, all results from Section~\ref{sec:clusterpointmeasures} therefore remain true under the combined assumption of $\mathcal D^{(m)}$ and \eqref{eq:thetacondition2}.
	
Now assume for some $p\in\NN$ that $C$ is a $p$-convex set. Then define the sequence $(C_n)_{n\in\NN}$ as
\[
	C_n= \sca  C,
\]
where $(\sca)_{n\in\NN}$ is a sequence of $d$-dimensional vectors $\sca = (\scal{1},\dots, \scal{d})$ such that $0 < \scal{\ell} \to\infty$, and where we without loss of generality will assume that $\abs{C}=1$. 
We now construct the index sets
\[
	D_n=C_n\cap\Zd
\] 
for all $n\in\NN$, and note in particular that $(D_n)_{n\in\NN}$ satisfies Assumption~\ref{ass:Cnassumption}
with $(\sca)$ playing the role of the scaling vector  from the assumption. Furthermore, $(x_n)$ and $(k_n)$ will be sequences, such that $(\xi_v)$, $(D_n)$, $(x_n)$ and $(k_n)$ jointly satisfy conditions specified in the relevant theorems.

As mentioned in the introduction we introduce two definitions of cluster counting processes, $N_n$ and $\tilde{N}_n$, respectively, counting the number of clusters (of the relevant type) of indices, where $\xi_v$ is above the level $x_n$ within $D_n$. First we define the measure $N_n$, based on the grid formed by the $J_z^{n,k_n}$-sets, and show Theorem~\ref{thm:pointprocesstheorem}. Afterwards, we introduce the more intuitive cluster measure $\tilde{N}_n$ and show a similar convergence theorem for that.

In the definition of the measure $N_n$, we count the set $J_z^{n,k_n}$ as a cluster, if $M_\xi(J_z^{n,k_n})>x_n$. For each $n$ we define the random measure $N_n$ on $C$ by counting the number of clusters in a scaled version of the index set as follows
\[
	N_n(A)=\sum_{z\in\Zd}\I{A}\Big(z\frac{\bb{t_{n,k_n}}}{\sca}\Big)
	\I{\{M_\xi(J_z^{n,k_n})>x_n\}}
\]
for all $A\in \Bb(C)$. Note that the cluster is counted as placed in the set $A$, if the down-scaled corner-point $z \,\bb{t_{n,k_n}}/\sca$ of the set $J_z^{n,k_n}$ is in $A$. We will later refer to such points, each representing a cluster, as \emph{cluster points}. It should be emphasized that $N_n$ depends crucially on the choice of the $k_n$-sequence.

In Theorem~\ref{thm:pointprocesstheorem} below we show that the sequence of random measures $(N_n)_{n\in\NN}$ \emph{converges in distribution with respect to the vague topology} towards a homogeneous Poisson measure $N$, denoted $N_n\stackrel{vd}{\to}N$. Details on this type of convergence can be found in \cite[Chapter~4]{Kallenberg2017}, and it is defined as
\[
\int_C f \dd N_n\to \int_C f \dd N\qquad\text{for all }f\in \hat{C}_C,
\]
where $\hat{C}_C$ denotes all bounded, continuous functions $f:C\to\RR_+$ with compact support. In the present case, where the limiting measure is a homogeneous Poisson measure, the convergence can equivalently be formulated as
\[
(N_n(A_1),\ldots,N_n(A_K))\stackrel{\mathcal{D}}{\to} (N(A_1),\ldots,N(A_K))
\]
as $n\to\infty$ for all $K\in\NN$ and all $A_1,\ldots,A_K\in \Bb(C)$ with $\abs{\partial A_k}=0$ for $k=1,\ldots,K$. See \cite[Theorem~4.11]{Kallenberg2017}.

\begin{theorem}\label{thm:pointprocesstheorem}
Let $(D_n)_{n\in\NN}$ be defined as above, and let $(\xi_v)_\vZ$ be a stationary field satisfying $\mathcal{D}(x_n;K_n)$ for some sequence $(x_n)_{n\in\NN}$. Moreover, let $k_n$ be a sequence of integers satisfying $k_n\to\infty$, $k_n \alpha_n \to 0$ and $k_n^{1/d} \bb{\gamma_n} = o(\sca)$ as $n\to\infty$, such that $\mathcal{D}^\ell(x_n;k_n; \theta)$ is satisfied for some $\theta\in (0,1]$. If for some $ 0 < \tau < \infty$,
\[
	\abs{D_n} \PP(\xi_0>x_n) \to \tau
\]
then $N_n\stackrel{vd}{\to} N$, where $N$ is a homogeneous Poisson measure on $C$ with intensity $\theta\tau$.
\end{theorem}

\begin{proof}
According to \cite[Theorem~4.18]{Kallenberg2017} it suffices to show that $\EE N_n(A)\to \EE N(A)$, where $A\subseteq C$ is a box on the form $A=\bigtimes_{j=1}^d(a_j,b_j]$, and that $\PP(N_n(B)=0)\to \PP(N(B)=0)$ for all $B$, where $B\subseteq C$ is a finite union of boxes each on the form $\bigtimes_{j=1}^d(a_j,b_j]$.

First we find that
\[
	\EE N_n(A)=\PP(M_\xi(J_0^{n,k_n})>x_n)
	\sum_{z\in\Zd}\I{A}\Big(z\frac{\bb{t_{n,k_n}}}{\sca}\Big).
\]
Recalling that 
\[
	t_{n,k_n}^\ell=\lfloor  c_{n,\ell}/k_n^{1/d}\rfloor,
\] 
it is easily seen that
\[
	\sum_{z\in\Zd}\I{A}\Big(z\frac{\bb{t_{n,k_n}}}{\sca}\Big)
	\sim k_n \abs{A}
\]
as $n\to\infty$. Combining this with (\ref{eq:knequality}), (\ref{eq:limitexpequality}) and (\ref{eq:thetacondition}) gives
\begin{align*}
	\EE N_n(A)&
	\sim \abs{A} k_n\PP(M_\xi(J_0^{n,k_n})>x_n)\\&
	\sim \abs{A}\abs{D_n}\PP(\xi_0>x_n)\PP(M_\xi(A_0^{n,k_n})\leq x_n\mid \xi_0>x_n)\\&
	\sim \abs{A}\theta \tau \\&
	= \EE N(A) .
\end{align*}
Now let $B\subseteq C$ be a finite union of boxes. Define $B_n=\sca B$ and furthermore
\[
	B_n^-=\bigcup_{z:I_z\subseteq B_n}I_z^{n,k_n},\qquad B_n^+=\bigcup_{z:I_z\cap B_n\neq\emptyset}I_z^{n,k_n},
\]
such that $B_n^-\subseteq B_n\subseteq B_n^+$.
By arguments as in Theorem~\ref{thm:maingeometrictheorem1}, $\abs{B_n^-}/\abs{B_n}\to 1$ and $\abs{B_n^+}/\abs{B_n}\to 1$, respectively.
Also note that both $B_n^-$ and $B_n^+$ themselves will be unions of (at most) the same number of boxes as $B$. Then the sequences $(B_n^-\cap\Zd)_{n\in\NN}$ and $(B_n^+\cap\Zd)_{n\in\NN}$ together with $(x_n)_{n\in\NN}$ and $(\xi_v)_{v\in\Zd}$ satisfy the conditions of Theorem~\ref{thm:maximumtheorem} with
\[
	\abs{B_n^-\cap\Zd}\,\PP(\xi_0>x_n)\to \abs{B}\tau
	\quad\text{and}\quad 
	\abs{B_n^+\cap\Zd}\,\PP(\xi_0>x_n)\to \abs{B}\tau,
\]
since $\abs{B_n\cap\Zd}/\abs{D_n}\to \abs{B}$. Then, by Theorem~\ref{thm:maximumtheorem},
\begin{align*}
	\PP(M_\xi(B_n^-\cap\Zd)\leq x_n)&\to \exp(- \abs{B}\theta \tau)=	\PP(N(B)=0),\\
	\PP(M_\xi(B_n^+\cap\Zd)\leq x_n)&\to \exp(- \abs{B}\theta \tau)=	\PP(N(B)=0).
\end{align*}
The desired limit follows, since
\[
	\PP(M_\xi(B_n^+\cap\Zd)\leq x_n)\leq \PP(N_n(B)=0)\leq \PP(M_\xi(B_n^-\cap\Zd)\leq x_n).
\]
This concludes the proof.
\end{proof}

Next we define the alternative cluster counting process $\tilde{N}_n$ based on what intuitively is interpreted as clusters. For this we define the set 
\[
	Q_{n,k_n}^-=
	\bigl\{z\in\Zd\::\: z\frac{\bb{t_{n,k_n}}}{\sca}\in C\bigr\}.
\]
That is all indices $z$ where $J_z^{n,k_n}$ can be counted by $N_n$. Note that $P_{n,k_n}\subseteq Q_{n,k_n}^-\subseteq Q_{n,k_n}$. Furthermore, let $\tilde{D}_{n,k_n}$ be the union of the corresponding $J_z^{n,k_n}$-sets
\[
	\tilde{D}_{n,k_n}=\bigcup_{z\in Q_{n,k_n}^-}J_z^{n,k_n}.
\]
Now we consider the rescaled set of indices
\[
	\Phi_n=\{v/\sca\::\: v\in \tilde{D}_{n,k_n}\:,\: \xi_v>x_n\},
\]
where the field is above $x_n$. Note that both of the measures $N_n$ and $\tilde{N}_n$ are based on variables $\xi_v$ in the same extended index set $\tilde{D}_{n,k_n}$. To define $\tilde{N}_n$, we divide $\Phi_n$ into a number of disjoint clusters of points such that different clusters are separated by a certain distance. More precisely, we say that $u,u'\in \Phi_n$ are in the same cluster, if there is a sequence of distinct elements $u=u_0,u_1\ldots,u_R=u'$ in that cluster such that
\[
\abs{u_i-u_{i-1}}\leq \frac{\sqrt{d}}{k_n^{1/d}}
\]
for all $i=1,\ldots,R$. The distance $\sqrt{d}/k_n^{1/d}$ is asymptotically equivalent to the largest distance within the rescaled boxes $\bb{c_n^{-1}}J_z^{n,k_n}$, and therefore the clusters produced this way is somewhat comparable to those counted by $N_n$. The procedure uniquely divides $\Phi_n$ into $X_{n,k_n}$ disjoint clusters, where $0\leq X_{n,k_n}\leq \abs{\Phi_n}$. If $X_{n,k_n}\geq 1$, let these clusters be denoted $\mathcal{C}_i^{n,k_n}$ for $i=1,\ldots,X_{n,k_n}$. The ordering is arbitrary and will not be relevant subsequently. For each cluster, we define the cluster point $x_i^{n,k_n}$, meaning the point in $\Rd$ that represents the cluster, as the point in $\mathcal{C}_i^{n,k_n}$ closest to the mean of the points in the cluster, i.e.
\[
x_i^{n,k_n}=\argmin_{x\in\mathcal{C}_i^{n,k_n}}\sum_{x'\in \mathcal{C}_i^{n,k_n}}\abs{x-x'}^2.
\]
In principle, we could have used any systematically chosen point from $\mathcal{C}_i^{n,k_n}$ as the cluster point representing the cluster --- the above is just an intuitively natural choice. Alternatively, we could have used the actual mean of the points in $\mathcal{C}_i^{n,k_n}$, but it will be helpful in the subsequent arguments that the cluster point itself corresponds to (the rescaling of) an extremal observation.

Based on the cluster points $x_1^{n,k_n},\ldots,x_{X_{n,k_n}}^{n,k_n}$ we define a random point measure on $C$ as
\[
\tilde{N}_n(A)=\sum_{i=1}^{X_{n,k_n}}\I{A}(x_i^{n,k_n})
\]
for all $A\in \Bb(C)$ with the convention that $\tilde{N}_n(A)=0$ if $X_{n,k_n}=0$. Note that similar to $N_n$, the measure $\tilde{N}_n$ will also depend on the sequence $(k_n)$. 

In Theorem~\ref{thm:pointprocesstheoremv2} below we see that $(\tilde{N}_n)_{n\in\NN}$ converges in exactly the same way as $(N_n)_{n\in\NN}$. The proof relies on finding a set with sufficiently large probability, where the two measures $N_n$ and $\tilde{N}_n$ are  identical on the sets $A$ and $B$ under study in the proof of Theorem~\ref{thm:pointprocesstheorem}.

\begin{theorem}\label{thm:pointprocesstheoremv2}
Let $(D_n)_{n\in\NN}$ be defined as above, and let $(\xi_v)_\vZ$ be a stationary field satisfying $\mathcal{D}(x_n;K_n)$ for some sequence $(x_n)_{n\in\NN}$. Moreover, let $k_n$ be a sequence of integers satisfying $k_n\to\infty$, $k_n \alpha_n \to 0$ and $k_n^{1/d} \bb{\gamma_n} = o(\sca)$ as $n\to\infty$, such that $\mathcal{D}^\ell(x_n;k_n; \theta)$ is satisfied for some $\theta\in (0,1]$. If for some $ 0 < \tau < \infty$,
\[
	\abs{D_n} \PP(\xi_0>x_n) \to\tau
\]
then $\tilde{N}_n\stackrel{vd}{\to} N$, where $N$ is a homogeneous Poisson measure on $C$ with intensity $\theta\tau$.
\end{theorem}
\begin{proof}
Similarly to the proof of Theorem~\ref{thm:pointprocesstheorem} it suffices to show that $\EE \tilde{N}_n(A)\to \EE N(A)$, where $A\subseteq C$ is a box on the form $A=\bigtimes_{j=1}^d(a_j,b_j]$ and that $\PP(\tilde{N}_n(B)=0)\to \PP(N(B)=0)$ for all $B$, where $B\subseteq C$ is a finite union of boxes on the form $\bigtimes_{j=1}^d(a_j,b_j]$.

Let $m\in\NN$ and define the events
\begin{align*}
	A_m&= \{ N(\partial A\oplus [-1/m,1/m]^d)=0\},\\
	A_{m,n}&= \{ N_n(\partial A\oplus [-1/m,1/m]^d)=0\},\\
	B_m&= \{N(\partial B\oplus [-1/m,1/m]^d)=0\},\\
	B_{m,n}&= \{N_n(\partial B\oplus [-1/m,1/m]^d)=0\},
\end{align*}
with the convention that $N_n$ and $N$ are 0 outside of $C$. 
These sets represent that there is no activity close (with $m$ large) to the boundaries of the sets $A$ and $B$, and they will help to ensure that clusters counted inside $A$ and $B$ by $N_n$ will also be counted as inside $A$ and $B$ by $\tilde{N}_n$. Note that $\PP(A_m^c)\to 0$ and $\PP(B_m^c)\to 0$ as $m\to\infty$.

Furthermore consider the finite collection $\mathcal{E}_m$ of overlapping boxes with side lengths $2/m$, indexed by the same $m\in\NN$ as above,
\[
	\mathcal{E}_m=
	\Bigl\{\frac{z}{m}+[-1/m,1/m]^d\::\: 
	\mathrm{dist}\Big(\frac{z}{m}, C\Big)\leq \frac{\sqrt{d}}{m},\, z\in\ZZ^d \Bigr\}.
\]
We define the events
\begin{align*}
E_m&=\{N(E)\leq 1\text{ for all }E\in \mathcal{E}_m\},\\
E_{m,n}&=\{N_n(E)\leq 1\text{ for all }E\in \mathcal{E}_m\},
\end{align*}
and notice that the overlapping nature of the sets in $\mathcal{E}_m$ ensures that no sufficiently small neighborhood in $C$ experiences more than one of the clusters counted by $N_n$.

Using the facts that the number of sets in $\mathcal{E}_m$ is of order $m^d$ and that $\PP(N(E)\geq 2)=O(1/m^{2d})$ together with Boole's inequality, it is seen that
\[
\PP(E_m^c)=O(1/m^d)
\]
as $m\to\infty$.

From Theorem~\ref{thm:pointprocesstheorem} we have $\PP(E_{m,n})\to \PP(E_{m})$ and $\PP(B_{m,n})\to \PP(B_{m})$ as $n\to\infty$, so for given $\delta>0$ we can choose $m,n_0\in\NN$ such that
\[
	\PP(B_{m,n}^c)< \delta \quad\text{and}\quad \PP(E_{m,n}^c)< \delta
\]
for all $n\geq n_0$. 

Note that any distance between two points within a set $\bb{c_n^{-1}}J_z^{n,k_n}$ is smaller than $\sqrt{d}/k_n^{1/d}$. Hence, choosing $n$ large enough ensures that for any set $\bb{c_n^{-1}}J_z^{n,k_n}$ the collection of all sets $\bb{c_n^{-1}}J_{z'}^{n,k_n}$ within a distance of at most $\sqrt{d}/k_n^{1/d}$ from $\bb{c_n^{-1}}J_z^{n,k_n}$ is fully contained in at least one set from $\mathcal{E}_n$.
For instance, choosing $n$ such that $(2\sqrt{d}+3)/k_n^{1/d}< 1/m$ suffices, and this choice will also be used in the proof of Theorem~\ref{thm:meancluster}.

With this choice of $n$ we have $N_n(A)=\tilde{N}_n(A)$ on $A_{m,n}\cap E_{m,n}$ and that $N_n(B)=\tilde{N}_n(B)$ on $B_{m,n}\cap E_{m,n}$: By definition of the two types of clusters any cluster point for $\tilde{N}_n$ will be counted as at least one cluster point for $N_n$. The set $E_{m,n}$ ensures that a cluster for $\tilde{N}_n$ is also at most one cluster for $N_n$. Furthermore, assuming the sets $A_{m,n}$ and $B_{m,n}$ gives that there is the same amount of clusters inside $A$ and $B$, respectively.

Now we easily see that
\[
	\PP(\tilde{N}_n(B)=0)\in 
	\bigl(\PP(N_n(B)=0)-2\delta, \,\PP(N_n(B)=0)+2\delta \bigr)
\]
for $n$ large, which together with Theorem~\ref{thm:pointprocesstheorem} gives the desired convergence of $\PP(\tilde{N}_n(B)=0)$ by letting $n\to\infty$ and then $\delta\to 0$. 

Next we turn to showing the convergence $\EE \tilde{N}_n(A)\to \EE N(A)$. First we find, using that $N_n(A)$ and $\tilde{N}_n(A)$ are equal on $A_{m,n}\cap E_{m,n}$ for $n$ large,
\begin{align*}
\liminf_{n\to\infty}\EE \tilde{N}_n(A)&\geq \liminf_{n\to\infty}\EE (\tilde{N}_n(A); A_{m,n}\cap E_{m,n})\\
&=\liminf_{n\to\infty}\EE (N_n(A); A_{m,n}\cap E_{m,n})=\EE (N(A); A_{m}\cap E_{m}).
\end{align*}
The last equality follows, since $\{N_n(A), N_n(E)\::\: E\in \mathcal{E}_m\}$ converges in distribution to $\{N(A), N(E)\::\: E\in \mathcal{E}_m\}$, such that also $(N_n(A), 1_{A_{m,n}\cap E_{m,n}})$ converges in distribution to $(N(A), 1_{A_{m}\cap E_{m}})$, and in addition $N_n(A)1_{A_{m,n}\cap E_{m,n}}\leq N_n(A)$, where $\EE N_n(A)\to \EE N(A)$. Letting $m\to\infty$ shows that $\liminf_{n\to\infty}\EE \tilde{N}_n(A)\geq \EE N(A)$. 

For the upper bound we utilize the inequality $\tilde{N}_n(A)\leq N_n(A\oplus [-1/m,1/m]^d)$ for $n$ large. Then, with similar considerations,
\begin{align*}
&\limsup_{n\to\infty}\EE \tilde{N}_n(A) \\
&\leq \limsup_{n\to\infty}\EE (N_n(A); A_{m,n}\cap E_{m,n})\\
&\phantom{=}+\limsup_{n\to\infty}\EE (N_n(A\oplus (-1/m,1/m]^d); A_{m,n}^c\cup E_{m,n}^c)\\
&=\EE (N(A); A_{m}\cap E_{m})+\EE (N(A\oplus(-1/m,1/m]^d); A_{m}^c\cup E_{m}^c) .
\end{align*}
Letting $m\to\infty$ together with integrability of $N(A)$ and $N(A\oplus (-1/m,1/m]^d)$ gives $\limsup_{n\to\infty}\EE \tilde{N}_n(A)\leq \EE N(A)$ as desired. 
\end{proof}

%----------------------------
	\section{Mean number of points in clusters}
	\label{sec:meannumberpoints}
%----------------------------

We continue considering the framework used to introduce the cluster processes $N_n$ and $\tilde{N}_n$ in Section~\ref{sec:clusterpointmeasures}. In particular, 
we define $(D_n)_{n\in\NN}$ as
\[
	D_n=(\sca C)\cap \Zd
\]
for a $p$-convex set $C\subseteq \Rd$ with $\abs{C}=1$. 
In this section we study the mean cluster size with respect to both of the cluster definitions. It should be noted that under the assumptions of Theorem~\ref{thm:pointprocesstheorem} and \ref{thm:pointprocesstheoremv2} the limit of the mean number of clusters within $C$ with respect to either definition is $\theta\tau$. The expected number of extremal points above $x_n$ within $D_n$ is (before rescaling)
\[
\sum_{v\in D_n}\PP(\xi_0>x_n),
\]
which by assumption converges to $\tau$ as $n\to\infty$. Thus one would expect that the mean number of extremal points within each cluster in the limit will be $1/\theta$. In Theorem~\ref{thm:meancluster} below, we show that this is indeed true. However, for the clusters counted by $\tilde{N}_n$ we need to impose a stronger mixing condition than $\mathcal{D}(x_n;K_n)$. Furthermore we demonstrate via conditioning that this limit is still valid independently of how many clusters there are in $C$ in total.

In relation to the cluster counting process $N_n$, where a set $J_z^{n,k_n}$ is counted as a cluster if $M_\xi(J_z^{n,k_n})>x_n$, we define the number of points in this potential cluster as
\[
	Y^{n,k_n}_z=\abs[\big]{\{v\in J_z^{n,k_n}\::\: \xi_v>x_n\}} .
\]
Clearly, $J_z^{n,k_n}$ represents a cluster if $Y_z^{n,k_n}>0$. We will be interested in the mean number of points in such a cluster, i.e.
\[
	\EE(Y_0^{n,k_n}\mid Y_0^{n,k_n}>0) .
\]

Similarly, for the cluster point measure $\tilde{N}_n$, there are clusters if $X_{n,k_n}=\tilde{N}_n(C)>0$. In that case let $\mathcal{C}$ be a cluster chosen uniformly among the $\tilde N_n(C)$ clusters. More specifically, we assume that $\mathcal{C}=\mathcal{C}^{n,k_n}_S$, where conditioned on $(\tilde{N}_n(C)=\ell)$ for each $\ell\in\NN$, the variable $S$ is uniform on $\{1,\ldots,\ell\}$ and independent of everything else.
Then we consider the mean cluster size defined as
\[
\EE(\abs{\mathcal{C}}\mid \tilde{N}_n(C)>0).
\]
In particular, we will focus on the limit of these conditional means as $n\to\infty$, and results for this are found in Theorem~\ref{thm:meancluster} below. We almost immediately have a result for the first of the two, while the second result requires the following stronger version of the mixing condition $\mathcal{D}(x_n;K_n)$.

\begin{cond4}
\label{cond:mixingcond2}
The condition $\overline{\mathcal{D}}(x_n;k_n;K_n)$ is satisfied for the stationary field $(\xi_v)_\vZ$ if there exists an increasing sequence $(\bb{\gamma_n})$ of $d$-dimensional vectors with $k_n^{1/d} \bb{\gamma_n} = o(\sca)$ such that the following holds: First, for all $\bb{\gamma_n}$-separated sets $A,B\subseteq K_n$,
\begin{equation}\label{eq:mixingcond21}
	\abs[\big]{
	\PP ((M_\xi(A)\leq x_n)\cap \mathcal{B}) -
	\PP (M_\xi(A)\leq x_n)
	\PP (\mathcal{B})
	}
	\le \alpha_n,
\end{equation}
where $\mathcal{B}\in\sigma((\xi_{v'}>x_n)\::\: v'\in B\})$. Secondly, for all $v\in K_n$ and  $B \subseteq K_n$, where $\{v\}$ and $B$ are $\bb{\gamma_n}$-separated,
\begin{equation}\label{eq:mixingcond22}
	\abs[\big]{
	\PP ((\xi_v>x_n)\cap \mathcal{B}) -
	\PP (\xi_v>x_n)
	\PP (\mathcal{B})
	}
	\le \tilde{\alpha}_n,
\end{equation}
where $\mathcal{B}\in\sigma((\xi_{v'}>x_n)\::\: v'\in B\})$. The sequences $(\alpha_n)$ and $(\tilde{\alpha}_n)$ satisfy
\[
	k_n\alpha_n\to 0
	\qquad\text{and}\qquad 
	\tilde{\alpha}_n =o(\PP(\xi_0>x_n))
\]
as $n\to \infty$.
\end{cond4}

Clearly (\ref{eq:mixingcond21}) in itself constitutes a stronger requirement than (\ref{eq:mixingcond}) in condition $\mathcal{D}(x_n;K_n)$, while (\ref{eq:mixingcond22}) serves as an additional assumption not directly related to (\ref{eq:mixingcond21}). Both however, express approximate extremal independence.

\setcounter{example}{1}
\begin{example}[Continued]
If $(\xi_v)_{v\in\Zd}$ is $m$-dependent, then $\overline{\mathcal{D}}(x_n;k_n;K_n)$ is satisfied, similarly to how $\mathcal{D}(x_n;K_n)$ is satisfied.
\end{example}

For understanding the formulation of the following theorem, recall that $Q_{n,k_n}^-$ is defined as the set of all $z\in\Zd$, where the cluster point $z$ (as counted by $N_n$) of $J_z^{n,k_n}$ is inside $D_n$.

\begin{theorem}\label{thm:meancluster}
Let $(D_n)_{n\in\NN}$ be defined as above, and let $(\xi_v)_\vZ$ be a stationary field satisfying $\mathcal{D}(x_n;K_n)$ for some sequence $(x_n)_{n\in\NN}$. Moreover, let $k_n$ be a sequence of integers satisfying $k_n\to\infty$, $k_n \alpha_n \to 0$ and $k_n^{1/d} \bb{\gamma_n} = o(\sca)$ as $n\to\infty$, such that $\mathcal{D}^\ell(x_n;k_n; \theta)$ is satisfied for some $\theta\in (0,1]$. Assume for some $ 0 < \tau < \infty$ that
\[
	\abs{D_n} \PP(\xi_0>x_n) \to \tau.
\]
Then as $n\to\infty$
\begin{equation}\label{fml:meancluster1}
\EE(Y_0^{n,k_n}\mid Y_0^{n,k_n}>0)\to 1/\theta.
\end{equation}
If furthermore $\overline{\mathcal{D}}(x_n;k_n;K_n)$ is satisfied, then 
\begin{equation}\label{fml:meancluster2}
\EE(\abs{\mathcal{C}}\mid \tilde{N}_n(C)>0)\to 1/\theta
\end{equation}
as $n\to\infty$, and additionally 
\begin{equation}\label{fml:meancluster12}
	\max_{z\in Q_{n,k_n}^-}\abs[\big]{\EE(Y_z^{n,k_n}\mid Y_z^{n,k_n}>0,N_n(C)=\ell)- 1/\theta}\to 0
\end{equation}
and 
\begin{equation}\label{fml:meancluster22}
\EE(\abs{\mathcal{C}}\mid \tilde{N}_n(C)=\ell)\to 1/\theta
\end{equation}
for all $\ell\in\NN$.
\end{theorem}

Recall that $\overline{\mathcal{D}}(x_n;k_n;K_n)$ is a stronger requirement than $\mathcal{D}(x_n;K_n)$, so requiring $\overline{\mathcal{D}}(x_n;k_n;K_n)$ for \eqref{fml:meancluster2}--\eqref{fml:meancluster22} makes the assumption of $\mathcal{D}(x_n;K_n)$ superfluous. The proof of Theorem~\ref{thm:meancluster} is rather technical and is deferred to Appendix~\ref{appnB}.

\setcounter{example}{2}
\begin{example}[Continued]
Recall that the field $(\xi_v)_\vZ$ with $\xi_v=\max_{z\in v+B} Y_z$ introduced in Example~\ref{ex:clusterexample} has extremal index $1/\abs{B}$ with respect to $(D_n)$. Let $(x_n)$ be chosen such that $\abs{D_n} \PP(\xi_0>x_n) \to \tau$ for some $0<\tau<\infty$. Since the field is $m$-dependent for a sufficiently large $m$, the condition $\overline{\mathcal{D}}(x_n;k_n;K_n)$ is satisfied. Then, due to Theorem~\ref{thm:meancluster}, the limiting mean number of extremal points within a cluster will not surprisingly be $\abs{B}$ with respect to both cluster definitions.
\end{example}

%----------------------------
	\section{Cluster counting measure on original scale}
	\label{sec:originalscale}
%----------------------------

The results from Section~\ref{sec:clusterpointmeasures} can be used to assess the limiting behavior of the number of clusters that $(\xi_v)_{v\in D_n}$ has above the level $x_n$ in the rescaled set $\sca A$,
where $A$ is some fixed subset of $C$. Describing the number of clusters in more general sequences of sets, only satisfying Assumption~\ref{ass:Cnassumption}, is not obtained directly from Section~\ref{sec:clusterpointmeasures}. Such a result is instead part of Theorem~\ref{thm:originalscale} below.
	
Again, we assume that the sequence $(D_n)_{n\in\NN}$ has the specific form
\[
	D_n=(\sca C) \cap \Zd
\]
with $\abs{C}=1$, such that Assumption~\ref{ass:Cnassumption} is satisfied.
This will not in itself be the set sequence of interest, but it will contain a more general sequence satisfying Assumption~\ref{ass:Cnassumption}. As before we let $(k_n)$ and $(x_n)$ be sequences satisfying some specified conditions. We define a cluster counting measure $L_n $ on $\Zd$ as
\[
	L_n(A)= \sum_{z\in\Zd}\I{A\cap D_n} (\bb{t_{n,k_n}}z)
	\I{\{M_\xi(J_z^{n,k_n})>x_n\}} 
\]
for all $A\subseteq \Zd$. That is simply a transformation of the measure $N_n$ back to the original scale and now regarded as a measure on $\Zd$.

Recalling the definition of the cluster points 
$x_1^{n,k_n},\ldots,x^{n,k_n}_{X_{n,k_n}}$
used to construct $\tilde{N}_n$, we define the measure $\tilde{L}_n$ on $\Zd$ as
\[
\tilde{L}_n(A)=\sum_{i=1}^{X_{n,k_n}}\I{A\cap D_n}(\bb{c_n}x_i^{n,k_n})
\]
for $A\subseteq \Zd$, with the convention that $\tilde{L}_n(A)=0$ if $X_{n,k_n}=0$.

Note that, as opposed to the measures $N_n$ and $\tilde{N}_n$ defined in Section~\ref{sec:clusterpointmeasures} above, both $L_n$ and $\tilde{L}_n$ are measures on the original index set scale.

\begin{theorem}\label{thm:originalscale}
Let $(\xi_v)_\vZ$ be a stationary field satisfying $\mathcal{D}(x_n;K_n)$ for some sequence $(x_n)_{n\in\NN}$. Moreover, let $k_n$ be a sequence of integers satisfying $k_n\to\infty$, $k_n \alpha_n \to 0$ and $k_n^{1/d} \bb{\gamma_n} = o(\sca)$ as $n\to\infty$, such that $\mathcal{D}^\ell(x_n;k_n; \theta)$ is satisfied for some $\theta\in (0,1]$.

Let $(B_n^1)_{n\in\NN},\dots,(B_n^G)_{n\in\NN}$ be sequences of sets in $\Zd$, each satisfying Assumption~\ref{ass:Cnassumption}, such that $B_n^g\subseteq D_n$ for all $n$ and $g$, and $B_n^1,\ldots,B_n^G$ are pairwise disjoint. Assume furthermore that
\[
	\lim_{n\to\infty}\frac{\abs{B_n^g}}{\abs{D_n}}=b_g
\]
for each $g=1,\ldots,G$, where $0\leq b_g<\infty$. If for some $ 0 < \tau < \infty$,
\[
	\abs{D_n} \PP(\xi_0>x_n) \to \tau
\]
then 
\begin{equation}\label{fml:originalscaleresult}
	(L_n(B_n^1),\ldots,L_n(B_n^G))\stackrel{\mathcal{D}}{\to} (L^1,\ldots,L^G)
\end{equation}
and 
\begin{equation}\label{fml:originalscaleresulttilde}
(\tilde{L}_n(B_n^1),\ldots,\tilde{L}_n(B_n^G))\stackrel{\mathcal{D}}{\to} (L^1,\ldots,L^G)
\end{equation}
as $n\to\infty$, where $L^1,\ldots,L^G$ are independent random variables with each $L^g$ being Poisson distributed with parameter $b_g\theta\tau$. Here $b_g=0$ means that $\PP(L^g=0)=1$. 
\end{theorem}

\begin{proof}
We demonstrate the proof of \eqref{fml:originalscaleresult}. The proof of \eqref{fml:originalscaleresulttilde} follows identically replacing $N_n$ and $L_n$ by $\tilde{N}_n$ and $\tilde{L}_n$, respectively. 

In addition to the usual set constructions, the proof applies the set construction from Section~\ref{appnA2} in Appendix~\ref{appnA},
which is similar to that of Section~\ref{sec:geometry} but with  $k\in\NN$ fixed and with another tuning parameter. We also refer to Section~\ref{appnA2} in Appendix~\ref{appnA} for the definition of $\tilde t_{n,k}^\ell$ and $\tilde I_z^{n,k}$ below.

We apply the construction to the set $D_n$ with tuning $\lambda \equiv 1$. That is, 
\[
	\tilde t_{n,k}^\ell = \frac{c_{n,\ell}}{k^{1/d}} 
\] 
for all $\ell$, and we thus divide $\Rd$, and thereby also $\sca C$, into the sets $\tilde I_z^{n,k}$ indexed by $z\in\Zd$. Moreover, we let the lattice points be $\tilde J_z^{n,k}=\tilde I_z^{n,k}\cap \Zd$. Now we see that
\[
	L_n( \tilde J_z^{n,k})=N_n\Big(\frac{1}{\sca}\tilde I_z^{n,k}\Big) ,
\]
where in fact
\[
	\frac{1}{\sca}\tilde I_z^{n,k} = z /\sqrt[d]{k} + 
	\bigl[0,1/\sqrt[d]{k} \,\bigr)^d
\]
is a fixed set independent of $n$ with volume $1/k$.
Thus, we immediately have from Theorem~\ref{thm:pointprocesstheorem} that for any finite fixed collection of distinct indices $\{z_1,\ldots,z_K\}$, it holds that
\begin{equation*}%\label{fml:distconvergenceJsets}
(L_n(\tilde J_{z_1}^{n,k}),\ldots,L_n(\tilde J_{z_K}^{n,k}))\stackrel{\mathcal{D}}{\to} (N_{z_1},\ldots,N_{z_K}),
\end{equation*}
where $N_{z_1},\ldots,N_{z_K}$ are independent and Poisson distributed each with parameter $\theta\tau/k$. In particular, this holds for the finite collection of $\tilde J_z^{n,k}$-sets that cover $D_n$ across all values of $n$; see Lemma~\ref{lem:geometrysup}\ref{eq:geomsup3} in Appendix~\ref{appnA}.

Now consider the sequences of discrete sets $(B_n^1)_{n\in\NN},\ldots,(B_n^G)_{n\in\NN}$. 
If $b_g>0$, the set-construction in this proof is identical to a set-construction based on the set $B_n^g$ using the scaling vectors $\bb{b_n^g} = \sca\, b_g^{1/d}$ and with a tuning parameter $\lambda = b_g$.
In line with previous notation, we let $B_{n,k}^{g,-}$ be the union of $\tilde J_z^{n,k}$-sets contained in $B_n^g$, and we let $p_{n,k}^g$ be the number of such sets. If $b_g=0$ for some $g$, then $p_{n,k}^g=0$ for $n$ large enough, and hence $B_{n,k}^{g,-}=\emptyset$. Note that $B_{n,k}^{1,-},\dots,B_{n,k}^{G,-}$ are disjoint. Let furthermore $B_{n,k}^+$ be the union of all $\tilde J_z^{n,k}$-sets that cover all $B_n^g$ jointly for $g=1,\ldots,G$. Let $q_{n,k}$ be the number of such sets in this union and note that
$
	p_{n,k}^1 + \dots + p_{n,k}^G \le q_{n,k} .
$
Define $p_k^g=\liminf_{n\to\infty}p_{n,k}^g$ and $q_k=\limsup_{n\to\infty}q_{n,k}$. Note, since all $p_{n,k}^g,q_{n,k}\in \NN$, that all $p_k^g\leq p^g_{n,k}$ and $q_k\geq q_{n,k}$ for $n$ large enough. From Lemma~\ref{lem:geometrysup}\ref{eq:geomsup2} in Appendix~\ref{appnA} we now have
\begin{equation}\label{fml:geometriccounting}
	p_k^g \sim b_g k\quad\text{ and }\quad r_k=o(k)
\end{equation}
as $k\to\infty$, where $r_k=q_k-(p_k^1+\dots+p_k^G)$. Choose for each $n$, $k$ and $g$ a set $B_{n,k}^{g,--}\subseteq B_{n,k}^{g,-}$ consisting of exactly $p_k^g$ sets of the type $\tilde J_z^{n,k}$. Let $R_{n,k}=B_{n,k}^+\setminus\bigcup_{g=1}^G B_{n,k}^{g,--}$ and note that $R_{n,k}$ is a union of at most $r_{k}$ sets of the type $\tilde J_z^{n,k}$ for $n$ large enough.

For any fixed selection $z_1,\dots,z_{r_{k}}$ we have that
\[
	\PP\Big(L_n\Big(\bigcup_{i=1}^{r_k}\tilde J_{z_i}^{n,k}\Big)=0\Big)\to\exp\Big(-r_k \frac{\theta\tau}{k}\Big),
\]
which will tend to 1 as $k\to\infty$. Also note that for a similar union of at most $r_k$ of the sets $\tilde J_z^{n,k}$, the probability will only be larger. Since, across all values of $n$, there are only finitely many selections of indices $z$ with $\tilde J_z^{n,k}$ intersecting $D_n$, the set $R_{n,k}$ can only consist of finitely many different selections of $\tilde J_z^{n,k}$-sets, each selection being of at most $r_k$ sets, across all values of $n$. Therefore also
\[
	\liminf_{n\to\infty} \PP\bigl(L_n(R_{n,k})=0 \bigr)\to 1
\]
as $k\to\infty$.
Now let $\ell_1,\ldots,\ell_G\in\NN_0$ and $\epsilon>0$ be given and choose $k\in\NN$ large enough such that
\[
	\liminf_{n\to\infty}\PP \bigl(L_n(R_{n,k})=0 \bigr)>1-\epsilon .
\]
Then 
\begin{align*}
	\MoveEqLeft	
	\limsup_{n\to\infty}\PP\bigl(L_n(B_n^1)=\ell_1,\ldots,L_n(B_n^G)=\ell_G\bigr) \\ &
	\le \limsup_{n\to\infty}\PP\bigl(L_n(B_n^1)=\ell_1,\ldots,L_n(B_n^G)=\ell_G,L_n(R_{n,k})=0\bigr)+\epsilon\\&
	= \limsup_{n\to\infty}\PP\bigl(L_n(B_{n,k}^{1,--})=\ell_1,\ldots,L_n(B_{n,k}^{G,--})=\ell_G,L_n(R_{n,k})=0\bigr)+\epsilon\\&
	\le \limsup_{n\to\infty}\PP\bigl(L_n(B_{n,k}^{1,--})=\ell_1,\ldots,L_n(B_{n,k}^{G,--})=\ell_G\bigr)+\epsilon\\&
	\le \limsup_{n\to\infty}\sup \PP\bigl(L_n(V^1_{n,k})=\ell_1,\ldots,L_n(V^G_{n,k})=\ell_G\bigr)+\epsilon\\&
	=\PP(L_{1,k}^-=\ell_1) \cdots \PP(L_{G,k}^-=\ell_G)+\epsilon,
\end{align*}
where $\sup$ in the fifth line is over all choices of disjoint sets $V^1_{n,k},\ldots,V^G_{n,k}$ such that each $V^g_{n,k}$ consists of $p_k^g$ (with fixed indices across varying $n$) sets $\tilde J_z^{n,k}$. Note that this is a finite supremum. In the sixth line each $L_{g,k}^-$ is Poisson distributed with parameter $p_k^g \theta\tau/k$. By similar considerations it is obtained that 
\begin{align*}
	\MoveEqLeft
	\liminf_{n\to\infty}\PP \bigl(L_n(B_n^1)=\ell_1,\ldots,L_n(B_n^G)=\ell_G \bigr)\\&
	\ge \PP(L_{1,k}^-=\ell_1)\cdots \PP(L_{G,k}^-=\ell_G)-\epsilon.
\end{align*}
Letting $k\to\infty$ now gives the desired result since each $L_{g,k}^-\stackrel{\mathcal{D}}{\to} L^g$ due to \eqref{fml:geometriccounting}. 
\end{proof}

%----------------------------
	\section{The special case of $\theta=1$}
	\label{sec:noclustering}
%----------------------------

In this section we again consider the sequence $(D_n)$ of discrete index sets obtained as 
\[
	D_n = (\sca C)\cap \Zd
\]
for some $p$-convex set $C$ with volume 1. 
Moreover, we restrict attention to the case $\theta=1$ thus allowing no clustering. We formulate all results under the assumption of the condition $\mathcal D^\ell(x_n;k_n;1)$ but recall from Lemma~\ref{lem:conditionrelation}\ref{eq:conditionrelation2} that it is equivalent to the assumption that $\mathcal D^{(m)}(x_n;k_n)$ and \eqref{eq:thetacondition2} are satisfied with $\theta=1$ for some (equivalently all) $m\in\NN_0$. 

We consider the classical point process of exceedances given by
\[
	\NNN_n(A) = \sum_{z\in\Zd}
	\I{A}\Bigl(\frac{z}{\sca}\Bigr)\I{\{\xi_z>x_n\}}
\]
for all $A\in \Bb(C)$. Not surprisingly, this converges similarly as the point process $N_n$ in Section~\ref{sec:clusterpointmeasures} though without clustering in the limiting process. The result follows by similar (in fact simpler) arguments as Theorem~\ref{thm:pointprocesstheorem}.

\begin{theorem}\label{thm:pointprocessexceedances}
Let $(D_n)_{n\in\NN}$ be defined as above, and let $(\xi_v)_\vZ$ be a stationary field satisfying $\mathcal{D}(x_n;K_n)$ for some sequence $(x_n)_{n\in\NN}$. Moreover, let $k_n$ be a sequence of integers satisfying $k_n\to\infty$, $k_n \alpha_n \to 0$ and $k_n^{1/d} \bb{\gamma_n} = o(\sca)$ as $n\to\infty$, such that $\mathcal{D}^\ell(x_n;k_n; 1)$ is satisfied. If for some $ 0 < \tau < \infty$,
\[
	\abs{D_n} \PP(\xi_0>x_n) \to \tau 
\]
then $\NNN_n\stackrel{vd}{\to} \NNN$, where $\NNN$ is a homogeneous Poisson process on $C$ with intensity $\tau$.
\end{theorem}

This result immediately gives the asymptotic behavior of upper order statistics of the field $(\xi_v)_\vZ$ as formulated in the corollary below. In the corollary, we let
\[
	\xi_{(1)}^n \ge \xi_{(2)}^n\ge \dots \ge  \xi_{(\abs{D_n})}^n 
	,\qquad n\in\NN,
\] 
be the ordered sample of $(\xi_v)_{v\in D_n}$ for all $n\in\NN$. In particular, $M_\xi(D_n) = \xi_{(1)}^n$.

\begin{corollary}
Let $(D_n)_{n\in\NN}$ be defined as above, and let $(\xi_v)_\vZ$ be a stationary field satisfying $\mathcal{D}(x_n;K_n)$ for some sequence $(x_n)_{n\in\NN}$. Moreover, let $k_n$ be a sequence of integers satisfying $k_n\to\infty$, $k_n \alpha_n \to 0$ and $k_n^{1/d} \bb{\gamma_n} = o(\sca)$ as $n\to\infty$, such that $\mathcal{D}^\ell(x_n;k_n; 1)$ is satisfied. If for some $ 0 < \tau < \infty$,
\[
	\abs{D_n} \PP(\xi_0>x_n) \to \tau ,
\]
then, for all $k$,
\[
	\PP\bigl(\xi_{(k)}^n \le x_n \bigr)
	\to \exp(-\tau) \sum_{j=0}^{k-1} \frac{\tau^j}{j!}
\]
as $n\to\infty$.
\end{corollary}

\begin{proof}
The claims follows directly from Theorem~\ref{thm:pointprocessexceedances} when realizing that
\[
	\PP\bigl(\xi_{(k)}^n \le x_n \bigr)
	= \PP \bigl( \NNN_n(C) \le k-1 \bigr)
\]
for all $n\in\NN$.
\end{proof}

Naturally, a version of Theorem~\ref{thm:originalscale} for an exceedance point process on the original index set $\Zd$ also holds in the present case of $\theta=1$. To this end define the point process
\[
	\overline L_n(A) = \sum_{z \in\Zd} \I{A \cap D_n} (z) \I{\{ \xi_z>x_n\}}
\]
for all $A\subseteq \Zd$. The result below now follows exactly as Theorem~\ref{thm:originalscale} replacing $N_n$ and $L_n$ by $\overline N_n$ and $\overline L_n$, respectively.

\begin{theorem}\label{thm:originalscaletheta1}
Let $(\xi_v)_\vZ$ be a stationary field satisfying $\mathcal{D}(x_n;K_n)$ for some sequence $(x_n)_{n\in\NN}$. Moreover, let $k_n$ be a sequence of integers satisfying $k_n\to\infty$, $k_n \alpha_n \to 0$ and $k_n^{1/d} \bb{\gamma_n} = o(\sca)$ as $n\to\infty$, such that $\mathcal{D}^\ell(x_n;k_n; 1)$ is satisfied.

Let $(B_n^1)_{n\in\NN},\dots,(B_n^G)_{n\in\NN}$ be sequences of sets in $\Zd$, each satisfying Assumption~\ref{ass:Cnassumption}, such that $B_n^g\subseteq D_n$ for all $n$ and $g$, and $B_n^1,\ldots,B_n^G$ are pairwise disjoint. Assume furthermore that
\[
	\lim_{n\to\infty}\frac{\abs{B_n^g}}{\abs{D_n}}=b_g,
\]
for each $g=1,\ldots,G$, where $0\leq b_g<\infty$. If for some $ 0 < \tau < \infty$,
\[
	\abs{D_n} \PP(\xi_0>x_n) \to \tau
\]
then 
\begin{equation*}%\label{fml:originalscaleresult}
	(L_n(B_n^1),\ldots,L_n(B_n^G))\stackrel{\mathcal{D}}{\to} (\overline L^1,\ldots,\overline L^G)
\end{equation*}
as $n\to\infty$, where $\overline L^1,\ldots,\overline L^G$ are independent random variables with each $\overline L^g$ being Poisson distributed with parameter $b_g\tau$. Here $b_g=0$ means that $\PP(\overline L^g=0)=1$. 
\end{theorem}

%%%%%%%%%%%%%%%%%%%%%%%%%%%%%%%%%%%%%%%%%%%%%%
%% Single Appendix:                         %%
%%%%%%%%%%%%%%%%%%%%%%%%%%%%%%%%%%%%%%%%%%%%%%
\begin{appendix}
\section{Geometric results and proofs}\label{appnA}
\begin{lemma}\label{lem:applemma1}
Let $B\subseteq \Zd$ be a finite set, and let $(D_n)$ be an increasing sequence of sets satisfying Assumption~\ref{ass:Cnassumption}. Then
\[
	\frac{\abs{D_n \oplus B}}{\abs{D_n}} \to 1
\]
as $n\to\infty$.
\end{lemma}

\begin{proof}
Let $C_n=\cup C_{n,i}$ be the continuous $p$-convex set associated to $D_n=C_n\cap \Zd$ by Assumption~\ref{ass:Cnassumption},
and let $\bb{c_n}$ be the $d$-dimensional scaling vector given in the assumption. Define $c_{n,*} = \min\{c_{n,\ell}\::\: \ell = 1,\dots,d\}$ for all $n$, which satisfies $c_{n,*} \to \infty$.  Let $R>0$ be given such that $B \subseteq B(R)$, that is, the closed ball of radius $R$.
Appealing to \cite[Corollary~1]{StehrRonnNielsen2020} (which was also used in the proof of Theorem~\ref{thm:maingeometrictheorem1}), and using that $\bb{c_n^{-1}} B(R) \subseteq B(R/c_{n,*})$ for all $n$, we find that
\begin{align*}
	\abs{\bb{c_n^{-1}}(\partial C_n \oplus B(R))} &
	\le \abs{\partial (\bb{c_n^{-1}}C_n) \oplus B(R/c_{n,*})} \\ &
	\le \sum_{i=1}^p \abs{\partial (\bb{c_n^{-1}}C_{n,i}) \oplus B(R/c_{n,*})} \\&
	\le 2\sum_{j=1}^{d-1} \omega_{d-j} \frac{R^{d-j}}{c_{n,*}^{d-j}}\sum_{i=1}^p V_j(\bb{c_n^{-1}}C_{n,i}) \\ &
	\to 0
\end{align*}
as $n\to\infty$, where the convergence follows from \eqref{ass:boundedintvolumes}. Using the fact that $\bb{c_n}$ satisfies $c_{n,1}\cdots c_{n,d} \sim \abs{C_n}$ as $n\to\infty$, we conclude that, 
\[
\begin{aligned}
	1 \le \frac{\abs{C_n \oplus B(R)}}{\abs{C_n}} &
	\le 1 + \frac{\abs{\partial C_n \oplus B(R)}}{\abs{C_n}} \\& 
	\sim 1 + \frac{\abs{\partial C_n \oplus B(R)}}{\prod c_{n,\ell}} \\ &
	= 1 + \abs{\bb{c_n^{-1}}(\partial C_n \oplus B(R))} \\ &
	\to 1
\end{aligned}
\] 
as $n\to\infty$. It can be seen that also $(C_n \oplus B(R)) \cap \Zd$ satisfies Assumption~\ref{ass:Cnassumption}, and therefore Theorem~\ref{thm:maingeometrictheorem1}\ref{eq:geomthm1new} applies to this set also. We therefore obtain that
\[
	1 \le \frac{\abs{D_n \oplus B}}{\abs{D_n}}
	\le \frac{\abs{(C_n \oplus B(R)) \cap \Zd}}{\abs{D_n}}
	\sim 
	\frac{\abs{C_n \oplus B(R)}}{\abs{C_n}}
	\to 1
\]
as $n\to\infty$. This concludes the proof.
\end{proof}

\subsection{Proof of Theorem~\ref{thm:maingeometrictheorem1}}\label{appnA1}

%\begin{proof}[Proof of Theorem~\ref{thm:maingeometrictheorem1}]
Throughout the proof we will use the boundedness \eqref{ass:boundedintvolumes} of the intrinsic volumes implied by Assumption~\ref{ass:Cnassumption}.

We start by demonstrating statement \ref{eq:geomthm2new}. This is more easily managed when considering all relevant sets scaled by $\bb{c_n^{-1}}$. In particular we will use that, in this scaling, the number of sets $\bb{c_n^{-1}} I_z^{n,k_n}$ contained in $\bb{c_n^{-1}}C_n$ and intersecting $\bb{c_n^{-1}} C_n$ by construction equals $p_{n,k_n}$ and $q_{n,k_n}$, respectively, from the original scaling.

Utilizing \cite[Corollary~1]{StehrRonnNielsen2020}, we find for any $0<r$ that
\begin{equation}\label{fml:tnklimit1}
\begin{aligned}
	\abs{\partial (\bb{c_n^{-1}}C_n)\oplus B(r\,k_n^{-1/d})} &
	\leq \sum_{i=1}^p\abs{\partial (\bb{c_n^{-1}}C_{n,i})\oplus B(r\,k_n^{-1/d})} \\&
	\leq 2\sum_{j=0}^{d-1} \omega_{d-j}\left(\sum_{i=1}^pV_{j}(\bb{c_n^{-1}} C_{n,i})\right) \frac{r^j}{k_n^{j/d}} \\ &
	\to 0
\end{aligned} 
\end{equation}
as $n\to\infty$, where $\partial$ denotes the boundary of the given set, and $\omega_j$ is the volume of the $j$-dimensional unit ball. To obtain the convergence, we used \eqref{ass:boundedintvolumes} and $k_n\to\infty$.
If $\bb{c_n^{-1}} I_z^{n,k_n}\cap \partial \bb{c_n^{-1}}C_n\neq \emptyset$ then necessarily 
$\bb{c_n^{-1}} I_z^{n,k_n}\subseteq \partial (\bb{c_n^{-1}}C_n)\oplus B\bigl(2\sqrt{d}\,k_n^{-1/d}\bigr)$ for sufficiently large $n$.
Using that
\[
	\abs{\bb{c_n^{-1}} I_z^{n,k_n}}
	=\prod_{\ell=1}^d \frac{t_{n,k_n}^\ell}{c_{n,\ell}}
	\sim \frac{1}{k_n} ,
\]
we then find, asymptotically,
\[
	\frac{q_{n,k_n}-p_{n,k_n}}{k_n}
	\le 
	\abs[\Big]{
	\bigcup_{z}
	\bb{c_n^{-1}} I_z^{n,k_n} }
	\le
	\abs*{\partial (\bb{c_n^{-1}}C_n)\oplus B\bigl(2\sqrt{d}\,k_n^{-1/d}\bigr)},
\] 
where the union is over all $z\in\Zd$ satisfying $\bb{c_n^{-1}} I_z^{n,k_n}\subseteq \partial (\bb{c_n^{-1}}C_n)\oplus B\bigl(2\sqrt{d}\,k_n^{-1/d}\bigr)$.
Together with \eqref{fml:tnklimit1} and the fact that $\limsup p_{n,k_n}/k_n \le 1 \le \liminf q_{n,k_n}/k_n$, this gives the statement in \ref{eq:geomthm2new}.

For statement \ref{eq:geomthm1new} we return to the original scaling of sets. We use that $\abs{\dnm}=p_{n,k_n} \prod t_{n,k_n}^\ell$ and $\abs{\dnp}=q_{n,k_n}\prod t_{n,k_n}^\ell$ together with \eqref{fml:Cnplusminus1} to find
\[
	\frac{p_{n,k_n}\prod t_{n,k_n}^\ell}{\abs{C_n}}
	\le 
	\frac{\abs{D_n}}{\abs{C_n}}
	\le 
	\frac{q_{n,k_n}\prod t_{n,k_n}^\ell}{\abs{C_n}}.
\]
Letting $n\to\infty$ and using the asymptotic equivalence of \ref{eq:geomthm2new} gives the desired result.

Statement \ref{eq:geomthm4new} easily follows by the assumption \eqref{eq:geoassumption1} (with a possibly different $c$).

Assume instead of \eqref{eq:geoassumption1} that $C_n$ is connected and that \eqref{ass:boundedintvolumes} is satisfied. Then the following shows that statement \ref{eq:geomthm4new} holds, and in particular \eqref{eq:geoassumption1} and consequently Assumption~\ref{ass:Cnassumption} are satisfied. Now, set $k_n \equiv 1$ in the above set-constructions. Then clearly $p_{n,1}\le 1$ for $n$ large enough. By considerations as above we therefore find for such $n$ that, asymptotically,
\[
	q_{n,1} - 1 \le q_{n,1}-p_{n,1} 
	\le 
	\abs{\partial( \bb{c_n^{-1}} C_n)\oplus B(2\sqrt{d})} ,
\]
which is bounded in $n$ due to \eqref{fml:tnklimit1} (with $k_n=1$) and \eqref{ass:boundedintvolumes}. Thus, $(q_{n,1})_{n\in\NN}$ is bounded by a constant $\overline c$. Since $0\in C_n$, we have $0\in Q_{n,1}$. Furthermore, $C_n$ is connected, so $Q_{n,1}$ consists of at most $\overline c$ points that are pairwise neighbors. Therefore,
\[
	Q_{n,1}
	\subseteq [-\overline c \:,\: \overline c]^d\cap\ZZ^d  .
\]
Realizing that $t_{n,k_n}^\ell \le t_{n,1}^\ell \sim c_{n,\ell}$ for any $n\in\NN$, $\ell=1,\dots,d$ and any sequence $k_n\to\infty$ shows \ref{eq:geomthm4new} for some $c>0$. This concludes the proof.
%\end{proof}

\subsection{Additional geometrical approximation result}\label{appnA2}

We need a set-construction similar to that of Section~\ref{sec:geometry} only with fixed $k\in\NN$ and an additional tuning parameter. The construction is very similar to that of \cite{StehrRonnNielsen2020}, but we include the result below for completeness.
We let $(B_n)_{n\in\NN}$ be a sequence of sets satisfying Assumption~\ref{ass:Cnassumption}, with $\overline B_n$ being the continuously indexed $p$-convex set associated to $B_n=\overline B_n\cap\Zd$.
Furthermore, let $\bb{b_n}=(b_{n,1},\dots,b_{n,d})$ be the $d$-dimensional scaling vectors appearing in the assumption (there denoted $\bb{c_n}$).
For all $n\in\NN$, $k\in\NN$ and any tuning constant $\lambda \in (0,\infty)$, we define 
\[
	\tilde t_{n,k}^\ell=\frac{b_{n,\ell}}{(\lambda\, k)^{1/d}} 
	\qquad\text{for all }\ell=1,\dots,d ,
\] 
with the dependence on $\lambda$ being suppressed in the notation for convenience. In what follows, bare in mind the dependence on $\lambda$ though.  
We collect the entries $\tilde t_{n,k}^\ell$ in the vector $\bb{\tilde t_{n,k}}$ and construct the box $\tilde I_z^{n,k}$ as
\[
	\tilde I_z^{n,k} = \bb{\tilde t_{n,k}}\bigl(z + [0,1)^d \bigr)
	= \bigtimes_{\ell=1}^d\big[z_\ell \tilde t_{n,k}^\ell,
	(z_\ell+1)\tilde t_{n,k}^\ell\big)
\]
for all $z\in\Zd$. Similarly as previously, we let $\tilde J_{z}^{n,k} = \tilde I_{z}^{n,k} \cap \Zd$ denote the lattice points of $\tilde I_{z}^{n,k}$. Lastly, we define $\tilde P_{n,k}$ and $\tilde Q_{n,k}$ to be the set of indices for which $\tilde J_z^{n,k}$ is contained in $B_n$ and intersects $B_n$, respectively. It can be seen that 
\begin{equation}\label{eq:pandqtilde}
	\limsup_{n\to\infty} \abs{\tilde P_{n,k}}
	\le \lambda k
	\le \liminf_{n\to\infty} \abs{\tilde Q_{n,k}},
\end{equation}
which will be used shortly.

\begin{lemma}\label{lem:geometrysup}
Let $(B_n)_{n\in\NN}$ be a sequence of sets satisfying Assumption~\ref{ass:Cnassumption}. Then
\begin{enumerate}[label=\normalfont(\roman*)]
	\item \label{eq:geomsup2} the sets $\tilde P_{n,k}$ and $\tilde Q_{n,k}$ defined above satisfy that
	\[
		\limsup_{n\to\infty} \abs{\tilde P_{n,k}} \sim \lambda k \sim
		\liminf_{n\to\infty}\abs{\tilde Q_{n,k}}
	\]
as $k\to\infty$.
\item\label{eq:geomsup3} for each $k\in \NN$ there exists $c_k<\infty$ such that
	\[
	 \bigcup_{z\in \tilde Q_{n,k} } \tilde J_z^{n,k}
	 \subseteq \bigcup_{z\in 
	 	[-c_{k},c_{k}]^d\cap \Zd} \tilde J_z^{n,k}
	\]
	for all sufficiently large $n$.
\end{enumerate}
\end{lemma}

\begin{proof}
By geometric considerations similar to those that led to \eqref{fml:tnklimit1} in the proof of Theorem~\ref{thm:maingeometrictheorem1}, we find that
\begin{equation}\label{eq:geometricsupeq1}
	\limsup_{n\to\infty}\,
	\abs[\big]{\partial (\bb{b_n^{-1}}\overline B_n)\oplus B(\sqrt{d}\, \lambda^{-1/d} k^{-1/d})} \to 0
\end{equation}
as $k\to\infty$, where $B(r)$ denotes the closed ball of radius $r$. 
Now, if $\tilde J_z^{n,k}$ has index $z\in \tilde Q_{n,k}\setminus \tilde P_{n,k}$ then $\tilde I_z^{n,k} \cap \overline B_n \neq \emptyset$ and $\tilde I_z^{n,k} \not\subseteq \overline B_n$, and consequently, down-scaling by $\bb{b_n}$,
\[
	\bb{b_n^{-1}} \tilde I_z^{n,k}\cap \partial (\bb{b_n^{-1}}\overline B_n)
	\neq \emptyset
	\qquad \text{and}\qquad
	\bb{b_n^{-1}} \tilde I_z^{n,k} \not \subseteq \bb{b_n^{-1}}\overline B_n .
\]
In particular, for any such index $z\in\Zd$ we have that 
\[
	\bb{b_n^{-1}} \tilde I_z^{n,k}\subseteq \partial( \bb{b_n^{-1}}\overline B_n)
	\oplus B(\sqrt{d}\, \lambda^{-1/d} k^{-1/d}) .
\]
Using that $\abs{\bb{b_n^{-1}}\tilde I_z^{n,k}}=1/(\lambda k)$, and taking the union over $z\in\Zd$ satisfying the inclusion above, we therefore obtain
\[
	\frac{\abs{\tilde Q_{n,k}}-\abs{\tilde P_{n,k}}}{\lambda k}
	\le 
	\abs[\Big]{
	%\smashoperator{ 
	\bigcup_{z}
	%}
	\bb{b_n^{-1}} \tilde I_z^{n,k} }
	\le \abs[\big]{\partial (\bb{b_n^{-1}}\overline B_n)\oplus 
	B(\sqrt{d}\, \lambda^{-1/d} k^{-1/d})} .
\] 
Combining \eqref{eq:pandqtilde} with \eqref{eq:geometricsupeq1} shows statement \ref{eq:geomsup2}.

Statement \ref{eq:geomsup3} easily follows from the fact that there exists $0<c<\infty$ such that
\[
	\overline B_n \subseteq
		\bigtimes_{\ell=1}^d \bigl[-c\, b_{n,\ell}\:,\: c\, b_{n,\ell}] ,
\]
combined with the fact that $\tilde J_z^{n,k}$ is a discrete box of asymptotical side lengths $b_{n,\ell}/(\lambda k)^{1/d}$ for $\ell=1,\dots,d$.
\end{proof}

\section{Proof of Theorem~\ref{thm:meancluster}}\label{appnB}

We will make use of the event
\[
	E_{m,n}=\{N_n(E)\leq 1\text{ for all }E\in \mathcal{E}_m\}
\]
for $m,n\in\NN$, where $\mathcal{E}_m$ is the collection of sets
\[
	\mathcal{E}_m=
	\Bigl\{\frac{z}{m}+[-1/m,1/m]^d\::\: 
	\mathrm{dist}\Big(\frac{z}{m}, C\Big)\leq \frac{\sqrt{d}}{m},\, z\in\ZZ^d \Bigr\}.
\]
Both are originally defined in the proof of Theorem~\ref{thm:pointprocesstheoremv2}.

\begin{proof}[Proof of Theorem~\ref{thm:meancluster}]
First we find
\begin{equation}\label{fml:meanclusternumber}
\EE(Y_0^{n,k_n})=\EE\Big(\sum_{v\in J_0^{n,k_n}} \I{\{\xi_v>x_n\}}\Big)=\abs{J_0^{n,k_n}}\PP(\xi_0>x_n)\sim \frac{\abs{D_n}}{k_n}\PP(\xi_0>x_n).
\end{equation}
As in the proof of Theorem~\ref{thm:pointprocesstheorem}, we have from (\ref{eq:knequality}), (\ref{eq:limitexpequality}) and (\ref{eq:thetacondition}) that
\begin{equation}\label{fml:tailofJ0}
	k_n\PP(M_\xi(J_0^{n,k_n})>x_n)
	\sim \abs{D_n}\theta \PP(\xi_0>x_n) .
\end{equation}
Now \eqref{fml:meancluster1} follows directly from combining \eqref{fml:meanclusternumber} and \eqref{fml:tailofJ0} with the assumption $\abs{D_n}\PP(\xi_0>x_n)\to \tau$.

Next we assume $\overline{\mathcal{D}}(x_n;k_n;K_n)$ and turn to demonstrating (\ref{fml:meancluster12}). In fact, we show the more general result that
\begin{equation}\label{fml:meancluster12ell}
	\max_{z\in Q_{n,k_n}^-}
	\abs[\big]{\EE(Y_z^{n,k_n}\mid Y_z^{n,k_n}>0,N_n(C)\in\mathcal{L})- 1/\theta}\to 0
\end{equation}
for any fixed subset $\mathcal{L}\subseteq\NN$. For this set, we additionally use the notation
\[
\mathcal{L}-1=\{z-1\::\: z\in\mathcal{L}\}.
\]
Clearly, (\ref{fml:meancluster12}) follows by letting $\mathcal{L}=\{\ell\}$. Define for each $z\in Q_{n,k_n}^-$  the set
\[
\tilde{H}_z^{n,k_n}= \bigl\{
		u\in\Zd\::\:
		z_\ell t_{n,k_n}^\ell+\gamma_{n,\ell} \le u_\ell \le (z_\ell+1) t_{n,k_n}^\ell - 1 - \gamma_{n,\ell},\ 
		\text{for all }\ell=1,\dots,d
	\bigr\}.
\]
Note that $\tilde{H}_z^{n,k_n}$ is defined very similar to the previously defined set $H_z^{n,k_n}$. However, $\tilde{H}_z^{n,k_n}$ is $\bb{\gamma_n}$-separated from any set disjoint of $J_z^{n,k_n}$. 

Noting that $\abs{J_z^{n,k_n}}\tilde{\alpha}_n=o(1/k_n)$, we find from (\ref{eq:mixingcond22}) that
\begin{equation}\label{fml:meanclusternumber2}
\begin{aligned}
	&\max_{z\in Q_{n,k_n}^-}\abs[\Big]{\EE(Y_z^{n,k_n}\I{\{N_n(C)\in\mathcal{L}\}})-\frac{\tau \PP(N(C)\in\mathcal{L} -1)}{k_n}}\\
	&=\max_{z\in Q_{n,k_n}^-}\abs[\Big]{\sum_{v\in J_z^{n,k_n}}\PP(\xi_v>x_n,N_n(C)\in\mathcal{L})-\frac{\tau \PP(N(C)\in\mathcal{L} -1)}{k_n}}\\
	&=\max_{z\in Q_{n,k_n}^-}\abs[\Big]{\sum_{v\in \tilde{H}_z^{n,k_n}}\PP(\xi_v>x_n,N_n(C\setminus \bb{c_n^{-1}}J_z^{n,k_n})\in\mathcal{L}-1)-\frac{\tau \PP(N(C)\in\mathcal{L} -1)}{k_n}}\\
	&\phantom{=}+o(1/k_n)\\
	&=\max_{z\in Q_{n,k_n}^-}\abs[\Big]{\abs{\tilde{H}_z^{n,k_n}}\PP(\xi_0>x_n)\PP(N_n(C\setminus \bb{c_n^{-1}}J_z^{n,k_n})\in\mathcal{L}-1)-\frac{\tau \PP(N(C)\in\mathcal{L} -1)}{k_n}}\\
	&\phantom{=}+o(1/k_n)\\
	&=\max_{z\in Q_{n,k_n}^-}\abs[\Big]{\abs{J_z^{n,k_n}}\PP(\xi_0>x_n)\PP(N_n(C\setminus \bb{c_n^{-1}}J_z^{n,k_n})\in\mathcal{L}-1)-\frac{\tau \PP(N(C)\in\mathcal{L} -1)}{k_n}}\\
	&\phantom{=}+o(1/k_n),
\end{aligned}
\end{equation}
where the second and fourth equality follows since $\abs{J_z^{n,k_n}\setminus \tilde{H}_z^{n,k_n}} =o(\abs{J_z^{n,k_n}})$ and $\abs{J_z^{n,k_n}}\PP(\xi_0>x_n)$ is of order $1/k_n$ since $0<\tau<\infty$. 

Now fix $m\in\NN$ for the moment and assume $1/k_n^{1/d}<1/m$, which occurs when $n$ is large enough. For any $z\in Q_{n,k_n}^-$ we can find a set $E\in \mathcal{E}_m$ such that $\bb{c_n^{-1}}J_z^{n,k_n}\subseteq E$. Recall that $\mathcal{E}_m$ is finite, so
\[
	\{N_n(E)\::\:E\in\mathcal{E}_m\}\stackrel{\mathcal{D}}{\to} \{N(E)\::\:E\in\mathcal{E}_m\}
\] 
by Theorem~\ref{thm:pointprocesstheorem}. Moreover, by stationarity, $\PP(N(E)=0)\geq \PP(N(E^{0,m})=0)$ with equality for all $E\subseteq C$, where $E^{0,m}=[-1/m,1/m]^d$ (or another $\mathcal{E}_m$-set fully contained in $C$). Then
\begin{align*}
	\liminf_{n\to\infty}\min_{z\in Q_{n,k_n}^-}
	\PP(N_n(\bb{c_n^{-1}}J_z^{n,k_n})=0) &
	\geq \liminf_{n\to\infty}\min_{E\in \mathcal{E}_m} \PP(N_n(E)=0)\\&
	=\PP(N(E^{0,m})=0).
\end{align*}
Since $\PP(N(E^{0,m})=0) \to 1$ as $m\to\infty$, we deduce that
\begin{equation}\label{fml:CminusJlimit}
\begin{aligned}
	\MoveEqLeft
	\max_{z\in Q_{n,k_n}^-}\abs[\Big]{\PP(N_n(C\setminus \bb{c_n^{-1}}J_z^{n,k_n})\in\mathcal{L}-1)-\PP(N(C)\in\mathcal{L}-1)} \\&
	=\abs[\big]{\PP(N_n(C)\in \mathcal L-1)-\PP(N(C)\in \mathcal L-1)}+o(1)\\&
	=o(1)
\end{aligned}
\end{equation}
as $n\to\infty$. Using this together with $\abs{J_z^{n,k_n}}\sim \abs{D_n}/{k_n}$ and the assumption $\abs{D_n}\PP(\xi_0>x_n)\to \tau$, we now have from \eqref{fml:meanclusternumber2} that 
\begin{equation}\label{fml:meanclusternumber3}
	\max_{z\in Q_{n,k_n}^-}\abs[\Big]{\EE(Y_z^{n,k_n}\I{\{N_n(C)\in\mathcal{L}\}})-
	\frac{\tau}{k_n} \PP(N(C)\in\mathcal{L}-1)}=o(1/k_n).
\end{equation}
After a few straightforward manipulations we have from \eqref{eq:mixingcond21} that
\begin{equation}\label{fml:tailofJ0withmore}
\begin{aligned}
	&\max_{z\in Q_{n,k_n}^-}\abs[\Big]{\PP(M_\xi(J_z^{n,k_n})> x_n,N_n(C)\in\mathcal{L})-\frac{\theta\tau}{k_n}\PP(N(C)\in\mathcal{L}-1)}\\&
	=\max_{z\in Q_{n,k_n}^-}\abs[\Big]{\PP(M_\xi(J_z^{n,k_n})> x_n,N_n(C\setminus \bb{c_n^{-1}} J_z^{n,k_n})\in\mathcal{L}-1)-\frac{\theta\tau}{k_n}\PP(N(C)\in\mathcal{L}-1)}\\&
	=\max_{z\in Q_{n,k_n}^-}\abs[\Big]{\PP(M_\xi(\tilde{H}_z^{n,k_n})> x_n,N_n(C\setminus \bb{c_n^{-1}} J_z^{n,k_n})\in\mathcal{L}-1)-\frac{\theta\tau}{k_n}\PP(N(C)\in\mathcal{L}-1)}\\&
	\phantom{=}+o(1/k_n)\\&
	=\max_{z\in Q_{n,k_n}^-}\abs[\Big]{\PP(M_\xi(\tilde{H}_z^{n,k_n})> x_n)\PP(N_n(C\setminus \bb{c_n^{-1}} J_z^{n,k_n})\in\mathcal{L}-1)-\frac{\theta\tau}{k_n}\PP(N(C)\in\mathcal{L}-1)}\\&
	\phantom{=}+o(1/k_n)\\&
	=\max_{z\in Q_{n,k_n}^-}\abs[\Big]{\PP(M_\xi(J_z^{n,k_n})> x_n)\PP(N_n(C\setminus \bb{c_n^{-1}} J_z^{n,k_n})\in\mathcal{L}-1)-\frac{\theta\tau}{k_n}\PP(N(C)\in\mathcal{L}-1)}\\&
	\phantom{=}+o(1/k_n)\\&
	=\max_{z\in Q_{n,k_n}^-}\abs[\Big]{\PP(M_\xi(J_z^{n,k_n})> x_n)\PP(N(C)\in\mathcal{L}-1)-\frac{\theta\tau}{k_n}\PP(N(C)\in\mathcal{L}-1)}+o(1/k_n)\\&
	=o(1/k_n),
\end{aligned} 
\end{equation}
where the 
second and fourth equality follow by arguments similar to those leading to \eqref{fml:meanclusternumber2}, 
the fifth equality is a result of \eqref{fml:CminusJlimit} and the last equality is due to \eqref{fml:tailofJ0} and the assumption $\abs{D_n}\PP(\xi_0>x_n)\to \tau$. Combining \eqref{fml:meanclusternumber3} and \eqref{fml:tailofJ0withmore} gives \eqref{fml:meancluster12ell}.

To show the results (\ref{fml:meancluster2}) and (\ref{fml:meancluster22}) for the second cluster measure, we first demonstrate that
\begin{equation}\label{fml:meanclusterEset}
\limsup_{m\to\infty}\limsup_{n\to\infty}\max_{z\in Q_{n,k_n}^-}\EE(Y_z^{n,k_n}\I{E_{m,n}^c}\mid Y_z^{n,k_n}>0,N_n(C)\in \mathcal{L})=0,
\end{equation}
where $\mathcal{L}\subseteq \NN$ is again a fixed subset. Now let $z\in Q_{n,k_n}^-$. For each $m\in\NN$ we divide $\mathcal{E}_m$ into two sub-collections
\begin{align*}	
\mathcal{E}_{m,z,n}&=\{E\in\mathcal{E}_m\::\:\bb{c_n^{-1}}\bb{t_{n,k_n}}z\in E \},\\
\mathcal{F}_{m,z,n}&=\{E\in\mathcal{E}_m\::\:\bb{c_n^{-1}}\bb{t_{n,k_n}}z\notin E \}.
\end{align*}
That is the $E$-sets wherein the potential cluster $J_z^{n,k_n}$ will be counted by $N_n$ and the $E$-sets, wherein it is not counted, respectively. We note that the number of sets in $\mathcal{E}_{m,z,n}$ is bounded by some finite constant $R$ independent of $m,z,n$. Obviously,
\begin{equation}\label{fml:twosums}
\begin{aligned}
	\MoveEqLeft
	\EE(Y_z^{n,k_n}\I{E_{m,n}^c}\mid Y_z^{n,k_n}>0,N_n(C)\in\mathcal{L})\\
	= & \sum_{E\in \mathcal{E}_{m,n,z}}\EE(Y_z^{n,k_n}\I{\{N_n(E)\geq 2\}}\mid Y_z^{n,k_n}>0,N_n(C)\in\mathcal{L})\\
	&+\sum_{F\in \mathcal{F}_{m,n,z}}\EE(Y_z^{n,k_n}\I{\{N_n(F)\geq 2\}}\mid Y_z^{n,k_n}>0,N_n(C)\in\mathcal{L}) .
\end{aligned}
\end{equation}
For any $F\in \mathcal{F}_{m,z,n}$ it is obtained with similar arguments to the above that
\begin{align*}
	\EE(Y_z^{n,k_n}\I{\{N_n(C)\in\mathcal{L},N_n(F)\geq 2\}}) &
	\le \EE(Y_z^{n,k_n}\I{\{N_n(F)\geq 2\}}) \\ &
	=\frac{1}{k_n}\tau \PP(N(F)\geq 2)+o(1/k_n)
\end{align*}
uniformly in $F$ and $z$. Using \eqref{fml:tailofJ0withmore} we then find
\begin{equation}\label{fml:boundFsets}
\begin{aligned}
	\MoveEqLeft
	\limsup_{n\to\infty}\max_{z\in Q_{n,k_n}^-}\sum_{F\in\mathcal{F}_{m,z,n}}\EE(Y_z^{n,k_n}\I{\{N_n(F)\geq 2\}}\mid Y_z^{n,k_n}>0,N_n(C)\in\mathcal{L} )\\& 
	\le \sum_{E\in\mathcal{E}_m}\frac{\PP(N(E)\geq 2)}{\theta\,\PP(N(C)\in\mathcal{L}-1)},
\end{aligned}
\end{equation}
which converges to 0 as $m\to\infty$, using that $\PP(N(E)\geq 2)$ is of order $1/m^{2d}$ and $\abs{\mathcal{E}_m}$ is of order $m^d$. 
We similarly have that
\begin{align*}
	\EE(Y_z^{n,k_n}\I{\{N_n(C)\in\mathcal{L},N_n(E)\geq 2\}}) & 
	\le \EE(Y_z^{n,k_n}\I{\{N_n(E)\geq 2\}}) \\&
	=\frac{1}{k_n}\tau \PP(N(E)\geq 1)+o(1/k_n)
\end{align*}
uniformly in $F$ and $z$. Thus,
\begin{equation}\label{fml:boundEsets}
\begin{aligned}
	\MoveEqLeft
	\limsup_{n\to\infty}\max_{z\in Q_{n,k_n}^-}\sum_{E\in\mathcal{E}_{m,z,n}}\EE(Y_z^{n,k_n}\I{\{N_n(E)\geq 2\}}\mid Y_z^{n,k_n}>0,N_n(C)\in\mathcal{L} )\\&
	\leq R\max_{E\in\mathcal{E}_m}\frac{\PP(N(E)\geq 1)}{\theta\,\PP(N(C)\in\mathcal{L}-1)}.
\end{aligned}
\end{equation}
Since $\PP(N(E)\geq 1)$ is of order $1/m^d$, this converges to 0 as $m\to\infty$. Combining (\ref{fml:twosums}), (\ref{fml:boundFsets}) and (\ref{fml:boundEsets}) gives (\ref{fml:meanclusterEset}). 

Using that the number of extremal points in $\mathcal{C}$ are bounded by the total number of extremal points in $\bigcup_{z\in Q_{n,k_n}^-}J_z^{n,k_n}$, we find that
\begin{equation}\label{fml:EConEbounded}
\begin{aligned}
	&
	\EE(\abs{\mathcal{C}}\I{E_{m,n}^c}\I{\{\tilde{N}_n(C)=\ell\}})\\ &
	\leq\EE(\abs{\mathcal{C}}\I{E_{m,n}^c}\I{\{N_n(C)>0\}})\\&
	\leq \sum_{z\in Q_{n,k_n}^-}\EE(Y_z^{n,k_n}\I{E_{m,n}^c}\I{\{N_n(C)>0\}})\\&
	=\sum_{z\in Q_{n,k_n}^-}\EE(Y_z^{n,k_n}\I{E_{m,n}^c}\mid Y_z^{n,k_n}>0,N_n(C)>0)\PP(Y_z^{n,k_n}>0,N_n(C)>0)\\&
	\leq \abs{Q_{n,k_n}^-}\PP(Y_0^{n,k_n}>0)\max_{z\in Q_{n,k_n}^-}\EE(Y_z^{n,k_n}\I{E_{m,n}^c}\mid Y_z^{n,k_n}>0,N_n(C)>0).
\end{aligned}
\end{equation}
From $P_{n,k_n}\subseteq Q_{n,k_n}^-\subseteq Q_{n,k_n}$ and Theorem~\ref{thm:maingeometrictheorem1} we have $\abs{Q_{n,k_n}^-}\sim k_n$, so due to (\ref{fml:tailofJ0}), (\ref{fml:meanclusterEset}) and the fact that $\PP(\tilde{N}_n(C)=\ell)$ has a non-vanishing limit, we also have
\[
	\limsup_{m\to\infty}\limsup_{n\to\infty}\EE(\abs{\mathcal{C}}\I{E_{m,n}^c}\mid \tilde{N}_n(C)=\ell)=0.
\]

Now let $\delta>0$ and choose $m,n_0\in\NN$ such that 
\begin{align*}
	\EE(\abs{\mathcal{C}}\I{E_{m,n}^c}\mid \tilde{N}_n(C)=\ell)
	& <\delta\\
	\max_{z\in Q_{n,k_n}^-}\EE(Y_z^{n,k_n}\I{E_{m,n}^c}\mid Y_z^{n,k_n}>0,N_n(C)=\ell)
	& <\delta\\
\frac{2\sqrt{d}+3}{k_n^{1/d}}&< 1/m
\end{align*}
for $n\geq n_0$. We write
\begin{equation*}%\label{fml:expectationwithdelta}
\EE(\abs{\mathcal{C}}\mid \tilde{N}_n(C)=\ell)=\EE(\abs{\mathcal{C}}\I{E_{m,n}}\mid \tilde{N}_n(C)=\ell)+\EE(\abs{\mathcal{C}}\I{E_{m,n}^c}\mid \tilde{N}_n(C)=\ell),
\end{equation*}
where the second term is bounded from above by $\delta$ for $n\geq n_0$. On $E_{m,n}$ and due to the assumption $(2\sqrt{d}+3)/k_n^{1/d}<1/m$ we have with the same arguments as in the proof of Theorem~\ref{thm:pointprocesstheoremv2} that each of the clusters, as counted by $\tilde{N}_n$, are contained in a $J_z^{n,k_n}$-set. In particular, on the set $E_{m,n}\cap(N_n(C)=\ell)$ we have that
\[
\abs{\mathcal{C}}=Y_{z_S}^{n,k_n},
\]
where $z_S$ is chosen uniformly on $\{z\in Q_{n,k_n}^-\::\:Y_z^{n,k_n}>0\}$ independently of everything else. Technically, this means that we choose a uniform integer number from $\{1,\ldots,\ell\}$ and use an order on $\Zd$ to translate the integer back to a $z$-value. Thus, using that $\tilde{N}_n(C)=N_n(C)$ on $E_{m,n}$ and $\PP(\tilde{N}_n(C)=\ell)\sim \PP(N_n(C)=\ell)$, we have
\begin{align*}
	&
	\EE(\abs{\mathcal{C}}\I{E_{m,n}}\mid \tilde{N}_n(C)=\ell)\\&
	\sim \EE(\abs{\mathcal{C}}\I{E_{m,n}}\mid N_n(C)=\ell)\\	&
	=\EE(Y_{z_S}^{n,k_n}\I{E_{m,n}}\mid N_n(C)=\ell)+o(1)\\&
	=\EE(Y_{z_S}^{n,k_n}\mid N_n(C)=\ell)-\EE(Y_{z_S}^{n,k_n}\I{E_{m,n}^c}\mid N_n(C)=\ell)+o(1).
\end{align*}
Applying similar arguments to (\ref{fml:EConEbounded}) gives that also
\[
	\EE(Y_{z_S}^{n,k_n}\I{E_{m,n}^c}\mid N_n(C)=\ell)<\delta
\]
for $n\geq \tilde{n}_0$ by choosing $\tilde{n}_0$ large enough. We now find that
\begin{equation*}
\begin{aligned}
	&\EE(Y_{z_S}^{n,k_n}\I{\{N_n(C)=\ell\}})\\
	&=\EE(\EE(Y_{z_S}^{n,k_n}\I{\{N_n(C)=\ell\}}\mid (\xi_v)_{v\in\Zd}))\\
	&=\EE \Bigl(\frac{1}{\ell}\sum_{z\in Q_{n,k_n}^-} Y_z^{n,k_n} \I{\{Y_z^{n,k_n}>0,N_n(C)=\ell\}}\Bigr)\\
	&=\frac{1}{\ell}\sum_{z\in Q_{n,k_n}^-}\EE(Y_z^{n,k_n}\mid Y_z^{n,k_n}>0,N_n(C)=\ell)\PP(Y_z^{n,k_n}>0,N_n(C)=\ell),
\end{aligned}
\end{equation*}
where the second equality is due to the independence between $z_S$ and everything else on $(N_n(C)=\ell)$. Using the uniform convergences in (\ref{fml:meancluster12}) and (\ref{fml:tailofJ0withmore}) and that $\abs{Q_{n,k_n}^-}\sim k_n$, we find that
\[
\EE(Y_{z_S}^{n,k_n}\I{\{N_n(C)=\ell\}})=\frac{\tau}{\ell}\PP(N(C)=\ell-1)+o(1).
\]
Using that 
\[
\frac{\PP(N(C)=\ell-1)}{\PP(N(C)=\ell)}=\frac{\ell}{\theta\tau}
\]
and letting $\delta\to 0$ gives (\ref{fml:meancluster22}).

Finally, we turn to demonstrating (\ref{fml:meancluster2}). For this, recalling that always $\tilde{N}_n(C)\leq N_n(C)$, we follow the lines from (\ref{fml:EConEbounded}) and find for all $\ell_0\in\NN$ that
\begin{align*}
	&\EE(\abs{\mathcal{C}}\I{\{\tilde{N}_n(C)> \ell_0\}})\\
	&\leq \EE(\abs{\mathcal{C}}\I{\{N_n(C)> \ell_0\}})\\
	&\leq \sum_{z\in Q_{n,k_n}^-}\EE(Y_z^{n,k_n}\I{\{N_n(C)> \ell_0\}})\\
	&=\sum_{z\in Q_{n,k_n}^-}\EE(Y_z^{n,k_n}\mid Y_z^{n,k_n}>0,N_n(C)> \ell_0)\PP(Y_z^{n,k_n}>0,N_n(C)> \ell_0).
\end{align*}
Again using the uniform convergences in (\ref{fml:meancluster12}) and (\ref{fml:tailofJ0withmore}), we find that 
\begin{equation*}%\label{fml:EClimsup}
	\limsup_{n\to\infty}\EE(\abs{\mathcal{C}}\I{\{\tilde{N}_n(C)> \ell_0\}})
	\leq \tau \PP(N(C)> \ell_0 - 1),
\end{equation*}
which decreases to 0 as $\ell_0\to\infty$. Note that (\ref{fml:meancluster22}) equivalently writes that for every fixed $\ell\in\NN$,
\begin{equation}\label{fml:Enotconditioned}
	\EE(\abs{\mathcal{C}}\I{\{\tilde{N}_n(C)=\ell\}})
	\to \PP(N(C)=\ell)/\theta,
\end{equation}
since $\PP(\tilde{N}_n(C)=\ell)\to \PP(N(C)=\ell)$. Therefore,
\[
	\EE(\abs{\mathcal{C}}\I{\{\tilde{N}_n(C)>0\}})
	=\sum_{\ell=1}^{\ell_0}	\EE(\abs{\mathcal{C}}\I{\{\tilde{N}_n(C)=\ell\}})
	+\EE(\abs{\mathcal{C}}\I{\{\tilde{N}_n(C)> \ell_0\}}) ,
\]
where the first term converges to $\PP(0<N(C)\leq \ell_0)/\theta$ as $n\to\infty$ due to \eqref{fml:Enotconditioned}, and the second in the limit is bounded by $\tau \PP(N(C)> \ell_0-1)$. Letting $\ell_0\to\infty$ shows
\[
	\EE(\abs{\mathcal{C}}\I{\{\tilde{N}_n(C)>0\}})
	\to  \PP(N(C)>0)/\theta,
\]
from which the desired property \eqref{fml:meancluster2} follows.
\end{proof}

\end{appendix}

\bibliographystyle{imsart-number} % Style BST file
\bibliography{bibref.bib}       % Bibliography file (usually '*.bib')

%% or include bibliography directly:
% \begin{thebibliography}{}
% \bibitem{b1}
% \end{thebibliography}

\end{document}